\documentclass[reqno,12pt]{amsart}
\usepackage{setspace,tikz,xcolor,mathrsfs,listings,multicol}
\usepackage{rotating}
\usepackage[vcentermath]{youngtab}
\usepackage[margin=1in,includefoot,footskip=30pt]{geometry}
\usepackage{enumerate}
\usepackage{booktabs}
\usepackage[all,cmtip]{xy}
\usetikzlibrary{arrows,matrix}
\tikzset{tab/.style={matrix of math nodes,column sep=-.35, row sep=-.35,text height=7pt,text width=7pt,align=center,inner sep=2,font=\footnotesize}}

\usepackage[colorlinks=true, pdfstartview=FitV, linkcolor=blue, citecolor=blue, urlcolor=blue]{hyperref}

\newcommand{\g}{\mathfrak{g}}
\newcommand{\gl}{\mathfrak{gl}}

\newcommand{\so}{\mathfrak{so}}

\newcommand{\inner}[2]{\left\langle #1, #2 \right\rangle}
\newcommand{\iso}{\cong}

\newcommand{\abs}[1]{\lvert #1 \rvert}

\newcommand{\field}{\mathbf{k}} 
\newcommand{\imi}{\mathbf{i}} 
\newcommand{\qbigwedge}{\bigwedge\nolimits_q}

\DeclareMathOperator{\Cl}{Cl} 

\newcommand{\mcA}{\mathcal{A}}
\newcommand{\mcB}{\mathcal{B}}
\newcommand{\mcC}{\mathcal{C}}

\newcommand{\zz}{\mathbf{z}}

\newcommand{\ZZ}{\mathbb{Z}}

\newcommand{\RR}{\mathbb{R}}
\newcommand{\CC}{\mathbb{C}}

\definecolor{darkred}{rgb}{0.7,0,0} 
\newcommand{\defn}[1]{{\color{darkred}\emph{#1}}} 

\definecolor{UQgold}{RGB}{196, 158, 54} 
\definecolor{UQpurple}{RGB}{73, 7, 94} 
\definecolor{UMNgold}{RGB}{255,200,46} 
\definecolor{UMNmaroon}{RGB}{106,0,50} 

\usepackage{listings}
\lstdefinelanguage{Sage}[]{Python}
{morekeywords={False,sage,True},sensitive=true}
\definecolor{dblackcolor}{rgb}{0.0,0.0,0.0}
\definecolor{dbluecolor}{rgb}{0.01,0.02,0.7}
\definecolor{dgreencolor}{rgb}{0.2,0.4,0.0}
\definecolor{dgraycolor}{rgb}{0.30,0.3,0.30}

\theoremstyle{plain}
\newtheorem{thm}{Theorem}[section]
\newtheorem{lemma}[thm]{Lemma}
\newtheorem{lem}[thm]{Lemma} 

\newtheorem{prop}[thm]{Proposition}
\newtheorem{cor}[thm]{Corollary}
\theoremstyle{definition}
\newtheorem{dfn}[thm]{Definition}
\newtheorem{ex}[thm]{Example}

\newtheorem{rmk}[thm]{Remark} 
\numberwithin{equation}{section}

\usepackage[colorinlistoftodos]{todonotes}

\newcommand{\phdg}{\phantom{\dagger}}

\newcommand{\dg}{*}

\newcommand{\ppdcomm}[1][a]{[\psi_{#1}^\dagger, \psi_{#1}]}

\usepackage{tikz-cd}

\usepackage{verbatim}
\usepackage{cleveref} 

\crefalias{equation}{rel} 
\newcommand{\crefrangeconjunction}{--} 
\crefformat{ax}{axiom~(#2#1#3)}
\crefmultiformat{ax}{axioms~(#2#1#3)}{ and~(#2#1#3)}{, (#2#1#3)}{ and~(#2#1#3)}
\crefrangeformat{ax}{axioms~(#3#1#4)\crefrangeconjunction(#5#2#6)}
\crefformat{cor}{corollary~#2#1#3}
\crefformat{defn}{definition~#2#1#3}
\crefmultiformat{defn}{definitions~#2#1#3}{ and~#2#1#3}{, #2#1#3}{ and~#2#1#3}
\crefformat{diag}{diagram~(#2#1#3)}
\crefmultiformat{diag}{diagrams~(#2#1#3)}{ and~(#2#1#3)}{, (#2#1#3)}{ and~(#2#1#3)}
\crefformat{fig}{figure~#2#1#3}
\crefmultiformat{fig}{figures~#2#1#3}{ and~#2#1#3}{, #2#1#3}{ and~#2#1#3}
\crefformat{lem}{lemma~#2#1#3}
\crefmultiformat{lem}{lemmae~#2#1#3}{ and~#2#1#3}{, #2#1#3}{ and~#2#1#3}
\crefformat{prop}{proposition~#2#1#3}
\crefmultiformat{prop}{propositions~#2#1#3}{ and~#2#1#3}{, #2#1#3}{ and~#2#1#3}
\crefformat{rel}{relation~(#2#1#3)}
\crefrangeformat{rel}{relations~(#3#1#4)\crefrangeconjunction(#5#2#6)}
\crefmultiformat{rel}{relations~(#2#1#3)}{ and~(#2#1#3)}{, (#2#1#3)}{ and~(#2#1#3)}
\crefformat{rmk}{remark~#2#1#3}
\crefformat{thm}{theorem~#2#1#3}
\crefmultiformat{thm}{theorems~#2#1#3}{ and~#2#1#3}{, #2#1#3}{ and~#2#1#3}

\newcommand{\inv}[1]{#1^{-1}}

\DeclareMathOperator{\End}{End}
\DeclareMathOperator{\Hom}{Hom}

\newcommand{\Cln}{\Cl\left(\field^n \oplus (\field^n)^*\right)}

\newcommand{\lieg}{\mathfrak{g}}

\newcommand{\qinv}{q^{-1}}

\newcommand{\Uqg}{U_q(\lieg)}

\newcommand{\Uqdn}{U_q(\mathfrak{so}_{2n})}
\newcommand{\Uqglm}{U_q(\mathfrak{gl}_m)}
\newcommand{\Uqgln}{U_q(\mathfrak{gl}_n)}

\newcommand{\Uqodn}{U_q(\mathfrak{o}_{2n})}

\newcommand{\Uqprime}{U_q'(\mathfrak{so}_m)}

\begin{document}
\title[Quantum Clifford structure and representations]{On the structure and representation theory of $q$-deformed Clifford algebras}

\author[W.~Aboumrad]{Willie Aboumrad}
\address[W.~Aboumrad]{The Institute for Computational and Mathematical Engineering (ICME) at Stanford University}
\email{willieab@stanford.edu}
\urladdr{https://web.stanford.edu/~willieab}

\author[T.~Scrimshaw]{Travis Scrimshaw}
\address[T.~Scrimshaw]{Faculty of Science, Hokkaido University, 5 Ch\=ome Kita 8 J\=onishi, Kita Ward, Sapporo, Hokkaid\=o 060-0808}
\email{tcscrims@gmail.com}
\urladdr{https://tscrim.github.io/}


\thanks{T.S.~was partially supported by Grant-in-Aid for JSPS Fellows 21F51028.}

\begin{abstract}
We provide a generalized definition for the quantized Clifford algebra introduced by Hayashi using another parameter $k$ that we call the twist.
For a field of characteristic not equal to $2$, we provide a basis for our quantized Clifford algebra, show that it can be decomposed into rank $1$ components, and compute its center to show it is a classical Clifford algebra over the group algebra of a product of cyclic groups of order $2k$.
In addition, we characterize the semisimplicity of our quantum Clifford algebra in terms of the semisimplicity of a cyclic group of order $2k$ and give a complete set of irreducible representations.
We construct morphisms from quantum groups and explain various relationships between the classical and quantum Clifford algebras.
By changing our generators, we provide a further generalization to allow $k$ to be a half integer, where we recover certain quantum Clifford algebras introduced by Fadeev, Reshetikhin, and Takhtajan as a special case.
\end{abstract}

\maketitle



\section{Introduction}
\label{sec:introduction}

Clifford algebras (over $\CC$) are a certain class of superalgebras defined by a universal property involving a symmetric bilinear form (equivalently a quadratic form) that appear naturally in the study of particle systems to describe fermionic creation and annihilation operators.
We build a Clifford algebra $\Cl(V \oplus V^{\dg})$ by taking an $n$ dimensional vector space $V$ and its algebraic dual $V^{\dg}$ to be maximal isotropic subspaces under the natural symmetric bilinear form defined by $\psi_i^{\dg}(\psi_j) = \delta_{ij}$ for dual bases $\{\psi_i\}_{i=1}^n$ and $\{ \psi_i^{\dg}\}_{i=1}^n$ of $V$ and $V^{\dg}$ respectively.
We can also form the (s)pin group, which is the universal cover of the (special) orthogonal group and is used to describe the symmetries of fermions, from the Clifford algebra.
By taking $n \to \infty$ in a suitable way, we arrive at the (fermionic) Fock space that is used to define representations of the infinite rank general linear Lie algebra $\gl_{\infty}$ and the affine general linear Lie algebra $\widehat{\gl}_n$ through the boson-fermion correspondence.
More precisely, this comes from the spinor representation of the Clifford algebra $\Cl(V \oplus V^{\dg})$ that is constructed on the exterior algebra $\bigwedge V$, which is the unique (up to isomorphism) simple module of a Clifford algebra.
Taking $V$ to be the natural $\gl_n$ representation, the exterior power $\bigwedge^k V$ is the irreducible fundamental representation $V(\Lambda_k)$ of $\gl_n$.
Given the relation with the spin group, we also obtain representations of the special orthogonal Lie algebra $\so_n$, with a slight dependency on the parity of $n$.
See, \textit{e.g.},~\cite{FH91} for more details.

There is a $q$-deformation of the representation theory of a Lie algebra $\g$ through its corresponding (Drinfel'd--Jimbo) quantum group $U_q(\g)$.
We refer the reader to standard textbooks such as~\cite{chari_pressley_1994,Jantzen96,Kassel95} for more information.
There are analogous $q$-deformations of the exterior algebra that have been used to construct $U_q(\g)$-representations, such as~\cite{DF94,Hayashi90,JMMO91,JMO00,KMS95,Kwon14,LT96}.
One interesting development is the relationship between the global crystal basis, in the sense of Kashiwara~\cite{Kashiwara90,Kashiwara91}, and Hecke algebras at roots of unity first conjectured by Lascoux, Leclerc, and Thibon~\cite{LLT96} and later proven by Ariki~\cite{Ariki96}.
Some additional recent developments can be found in~\cite{Gerber19}.

In this paper, we focus on the construction of a $q$-deformation of the Clifford algebra given by Hayashi~\cite{Hayashi90} that we call the quantum Clifford algebra.
In~\cite{Kwon14}, Kwon uses the quantum Clifford algebra that is attributed to Hayashi, but the presentation is slightly different.
Indeed, the dimensions of the two algebras are different, although this is done to account for the difference in the quantum groups used in their constructions.
In this paper, we generalize these constructions to have a second parameter $k$ that we call the twist; the case $k = 1$ was used in~\cite{Kwon14} while the case $k = 2$ in~\cite{Hayashi90}.
Furthermore, since the Clifford algebra construction works over an arbitrary field $\field$ of chararacteristic not equal to $2$, we work at this level of generality.\footnote{Modifications can be made for the characteristic $2$ case, but the situation is drastically different.}
We denote the rank $n$ and twist $k$ quantum Clifford algebra as $\Cl_q(n, k)$, where $q \in \field \setminus \{0\}$.

In order to work with these algebras, we first construct a basis (\Cref{thm:basis_general}) analogous to the usual Clifford algebra basis through a series of reduction rules to a standard form.
This allows us to show the dimension of $\Cl_q(n,k)$ is equal to $(8k)^n$ (\Cref{clqnk dim}).
This recovers the remark at the end of~\cite{Hayashi90} due to M.~Takeuchi about the dimension of $\Cl_q(n, 2)$ over $\CC$.
In \Cref{sec:Clifford_relationship}, we expand this remark to our level of generality provided we have all $2k$th roots of unity in $\field$.
Using our basis, we are then able to compute the center of $\Cl_q(n,k)$ as the group algebra $\field [\ZZ_{2k}^n]$ of a product of cyclic groups of order $2k$ (\Cref{clqnk center}).
We also extend the morphisms from a quantum group (\Cref{twisted qgp into clqnk}) given in~\cite{Hayashi90}.

Since the quantum Clifford algebra is a superalgebra (that is, it has a $\ZZ_2$-grading), we can decompose it into a signed tensor product of rank $1$ Clifford algebras, so
\[
\Cl_q(n, k) \iso \Cl_q(1, k) \, \widehat{\otimes} \, \Cl_q(1,k) \, \widehat{\otimes} \,\cdots \, \widehat{\otimes} \, \Cl_q(1,k),
\]
analogous to the classical case.
However, it will be useful to decompose this as a usual tensor product of algebras (\Cref{clqn clqnm embedding}),
\[
\Cl_q(n, k) \iso \Cl_q(1, k) \otimes \Cl_q(1,k) \otimes \cdots \otimes \Cl_q(1,k),
\]
which is also analogous to the classical case.
Following the classical case, we construct this isomorphism by introducing a quantum analog of the volume element.
The volume element of the usual Clifford algebra is an important element as it can be used to prove Bott periodicity, as explained in \cite[Prop.~3.5, Thm.~3.7]{michelsohn_lawson}.
It also has important representation-theoretic implications as it can be used to classify irreducible representations of any Clifford algebra; see~\cite[Prop.5.9,~5.10]{michelsohn_lawson}.
In spin geometry, the volume element gives any Dirac bundle a $\ZZ_2$-grading \cite[Ch.~6]{michelsohn_lawson}, and in relativistic quantum mechanics, it is used to define chirality for spinors~\cite{trindade_floquet_vianna_2020}.

While $\Cl_q(n, k)$ behaves much like its classical counterpart, there are a number of important distinctions coming from the existence of a nontrivial center.
The first is that $\Cl_q(n,k)$ is no longer a simple algebra, but its semisimplicity is tied directly to the semisimplicity of its center (\Cref{generalized semisimplicity}).
This is a reflection of the fact that $\Cl_q(n,k)$ could be seen as a classical Clifford algebra over its center.
This also allows us to construct all irreducible representations, which we can easily characterize when we have all $2k$th roots of unity in $\field$ (\Cref{clqnk repn} and \Cref{clqnk is semisimple}).
Although our given proof uses the explicit structure from our basis and the (unsigned) tensor product decomposition.
Moreover, our construction of the irreducible representations uses the braided exterior algebra of~\cite{berenstein}, which is equivalent to the natural generalization of Hayashi's construction~\cite{Hayashi90} on the usual exterior algebra but is better behaved when taking tensor products of quantum group representations.

Next, by observing the relations, we note that almost all relations can be defined when $k \in \frac{1}{2}\ZZ_{>0}$.
We extend to this case by slightly changing our generating set in \Cref{sec:root}, which leads to analogous results in the half twist cases.
In particular, we show that the quantum Clifford algebra can recover the usual Clifford algebra when $k = \frac{1}{2}$ (\Cref{thm:classical_iso}).
Furthermore, it becomes easy to see that this is the quantum Clifford algebra construction given by Ding and Frenkel~\cite{DF94} based upon the FRT construction~\cite{FRT89}.
Thus, we show that our family of quantum Clifford algebras encompasses the seemingly different constructions of~\cite{DF94} and~\cite{Hayashi90}.

Let us discuss an application of our results to quantized skew Howe duality, which is a relationship between commutating actions of a quantum group and (a coideal subalgebra of) a possibly different quantum group on a braided exterior algebra~\cite{CKM14,LZZ11,QS19,ST19}.
The skew $\Uqgln \otimes \Uqglm$-duality~\cite[Theorem~3.19]{willie_a} (resp.\ skew $\Uqodn \otimes \Uqprime$-duality \cite[Theorem~3.19]{willie_bd}) relies on embeddings of $\Uqgln$ and $\Uqglm$ (resp. $\Uqdn$ and $\Uqprime$) into $\Cl_q(nm, 1)$.
These embeddings leverage the structure theorem \ref{clqn clqnm embedding} in order to define quantum group actions on certain tensor products of modules.

This paper is organized as follows.
In \Cref{cln props}, we recall some facts and properties of the classical Clifford algebra.
In \Cref{clqnk props}, we prove a number of structural properties of the quantum Clifford algebra.
In \Cref{clqnk rep theory}, we discuss the representation theory of the quantum Clifford algebra.
In \Cref{sec:root}, we discuss the half twist cases of the quantum Clifford algebra.
In \Cref{sec:Clifford_relationship}, we generalize the remark of Takeuchi in~\cite{Hayashi90}.

\subsection*{Acknowledgements}

The authors thank Daniel Bump and Jae-Hoon Kwon for useful discussions.
The first author also thanks Daniel Bump for his patient guidance and caring support throughout the years.
The second author thanks Stanford University for its hospitality during his visit in May, 2022.

This work benefited from computations using \textsc{SageMath}~\cite{sagemath}.
Some of the results in this work were discovered by the first author as part of his dissertation research.
This work was partly supported by Osaka City University Advanced Mathematical Institute (MEXT Joint Usage/Research Center on Mathematics and Theoretical Physics JPMXP0619217849).

\section{Background: Classical Clifford algebras}
\label{cln props}

In this section we define the classical Clifford algebra and recall some of its basic properties.

Every algebra considered here is unital and associative unless otherwise stated.
We fix a field $\field$ of characteristic different from $2$ and compute all tensor products over $\field$.
We denote by $\zeta_r$ a primitive $r$-th root of unity in $\field$ (assuming it exists); so $\zeta_r = e^{2\pi \imi / r}$ with $\imi = \sqrt{-1}$ in $\CC$.
We use $\ZZ_r = \ZZ / r \ZZ$ to denote the (cyclic) group of integers modulo $r$ under $+$, and let $e_j$ denote the $j$th standard basis vector in $\ZZ^n$.
For any ring $R$, we let $R^{\times}$ denote its group of units.
We use $\{A, B\}$ (resp.~$[A, B]_q$) to denote the algebra anticommutator (resp.\ $q$-commutator) of $A$ and $B$:
\[
\{A, B\} = AB + BA,
\qquad\qquad
[A, B]_q = AB - q \, BA.
\]

\begin{dfn}[{\cite[Def.~6.1.1]{GW}}] \label[defn]{classical clifford alg defn}
	Let $W$ be a vector space over $\field$ equipped with a symmetric bilinear form $\beta$.
	A \defn{Clifford algebra} for $(W, \beta)$ is an associative $\field$-algebra $\Cl(W, \beta)$ and a linear map $\gamma\colon W \hookrightarrow \Cl(W, \beta)$ satisfying the following conditions:
	\begin{enumerate}[(i)]
		\item $\{\gamma(v), \gamma(w)\} = \beta(v, w)$ for every $v, w \in W$.
		\item $\gamma(W)$ generates $\Cl(W, \beta)$ as an algebra.
		\item (Universal Property) Given any unital associative $\field$-algebra $\mcA$ with a linear map $\varphi\colon W \to \mcA$ such that $\{\varphi(v), \varphi(w)\} = \beta(v, w)$, there exists an algebra homomorphism $\widetilde{\varphi}\colon \Cl(W, \beta) \to \mcA$ such that $\varphi = \widetilde{\varphi} \circ \gamma$.
		In other words, there is an algebra map $\widetilde{\varphi}$ making the following diagram commute.
	\begin{equation*}
	\begin{tikzcd}
		W \ar[d, swap, "\gamma"] \ar[r, "\varphi"]
		& \mcA 
		\\
		Cl(W, \beta) \ar[ur, dashed, swap, "\widetilde{\varphi}"]
	\end{tikzcd}
	\end{equation*}
	\end{enumerate}
\end{dfn}

From the universal property, Clifford algebras are unique (up to isomorphism) for the pair $(W, \beta)$, and so we call $\Cl(W, \beta)$ the Clifford algebra of $(W, \beta)$.
We can explicitly construct it as a quotient of the tensor algebra $T(W) = \bigoplus_{m=0}^\infty W^{\otimes m}$ modulo the ideal generated by elements of the form $v \otimes v - \beta(v,v) 1_{T(W)}$.
Hence, the natural $\ZZ$-grading on $T(W)$ descends to a $\ZZ_2$-grading, making $\Cl(W, \beta)$ a \defn{superalgebra}.
The \defn{exterior algebra} is when $\beta = 0$ and is denoted $\bigwedge(W) = \Cl(W, 0)$, and it is supercommutative.

In this work, we mostly consider Clifford algebras for spaces of the form $V \oplus V^*$, where $V^* := \Hom_{\field}(V, \field)$ is the usual algebraic dual of $V$.
We equip $V \oplus V^*$ with the symmetric bilinear form $\beta$ resulting from the duality pairing between $V$ and $V^*$:
\begin{equation}\label{symm bilinear form}
	\beta\big((v, f), (w, h)\big) = f(w) + h(v), \quad v, w \in V, \,\, f, h \in V^*.
\end{equation}
Hence $V$ and $V^*$ are maximal isotropic subspaces with respect to $\beta$; that is, they are Lagrangian subspaces of $V \oplus V^*$.
For simplicity, we write $\Cl(V \oplus V^*) = \Cl(V \oplus V^*, \, \beta)$.
We will work in a fixed basis $\{\psi_a\}_{a=1}^n$ for $V$, and let $\{\psi_b^{\dg}\}_{b=1}^n$ denote the dual basis in $V^*$.
Then the $\psi_a$ and $\psi_b^{\dg}$ satisfy the \defn{canonical anticommutation relations}
\begin{align}\label{cac}
	\begin{gathered}
		\psi_a \psi_b + \psi_b \psi_a = \psi_a^\dg \psi_b^\dg + \psi_b^\dg \psi_a^\dg = 0, \\
		\psi_a \psi_b^\dg + \psi_b^\dg \psi_a = \delta_{ab}.
	\end{gathered}
\end{align}

Recall that $\Cl(V \oplus V^*)$ is a central simple algebra, and therefore it has a unique irreducible representation called the \defn{spinor representation}, up to isomorphism.
See, \textit{e.g.},~\cite[Thm.~31.1]{Bump}, \cite[Ch.~IV]{Knus91}, \cite[Ch.~5]{Lam05}, or~\cite[Ch.~V]{OMeara00} for details.
The irreducible $\Cl(V \oplus V^*)$-module may be realized on the exterior algebra $\bigwedge(V)$ by imposing that $\psi_i^* \cdot 1_{\bigwedge(V)} = 0$; equivalently, this is $\Cl(V \oplus V^*) / J^*$, where $J^*$ is the left ideal generated by~$V^*$.

We conclude this section by constructing an isomorphism
\[
\Gamma_V\colon \Cl(V \oplus V^*)^{\otimes m} \to \Cl\left(V \otimes W) \oplus (V \otimes W)^*\right),
\]
with $m = \dim W$ that we will generalize to the quantum setting in the sequel.
For each $r = 1, \ldots, n$, define the \defn{volume} or \defn{chirality element} of the subalgebra $\Cl(\field^r \oplus (\field^r)^*) \subset \Cl(V \oplus V^*)$ by
\begin{align}\label{classical volume elt}
\overline{f}_r = [\psi_1, \psi_1^\dg] \cdots [\psi_r, \psi_r^\dg].
\end{align}
This name comes from rewriting \Cref{classical volume elt} as
\[
\overline{f}_r = \epsilon_1 \cdots \epsilon_{2r},
\]
where
\begin{equation}\label{classical std coords}
\epsilon_{2j-1} = \psi_j^\dg - \psi_j
	\quad\text{and}\quad 
\epsilon_{2j} = \psi_j^\dg + \psi_j
\end{equation}
(\textit{cf.}~\cite{michelsohn_lawson}) since $\epsilon_{2j-1} \epsilon_{2j} = [\psi_j, \psi_j^\dg]$.
The $\epsilon_j$ correspond to the (standard) orthonormal basis of $\field^{2n}$ with the bilinear form defined by~\eqref{symm bilinear form}, which becomes a diagonal bilinear form $\delta$ in this basis so that $\Cl(V \oplus V^*) \iso \Cl(\field^{2n}, \delta)$ (see \Cref{eq:usual_Clifford_COB} and \Cref{eq:inverse_COB} for a precise example).
Hence, $\overline{f}_r$ is the volume element in the corresponding exterior algebra $\bigwedge(\field^{2r})$.
Contrast this with $\psi_j, \psi_j^\dg$, which define an isotropic basis of $\field^n \oplus (\field^n)^\dg$ with the polarized bilinear form of~\eqref{symm bilinear form}.

In this context we use the volume element to construct a superalgebra isomorphism expressing a tensor product of Clifford algebras as another Clifford algebra.
This map is motivated by~\cite[Lemma~1.2]{wenzl_spin_centralizer}.

\begin{prop}\label[prop]{cl tensor m into clnm}
	Suppose $V$ and $W$ are $\field$-vector spaces with $\dim V = n$ and $\dim W = m$. Let $\phi_a$ denote either $\psi_a$ or $\psi_a^\dg$.
	There is a superalgebra isomorphism $\Gamma_V\colon \Cl(V \oplus V^*)^{\otimes m} \to \Cl\left((V \otimes W) \oplus (V \otimes W)^*\right)$ satisfying
	\begin{align*}
		1 \otimes \cdots \otimes \phi_a \otimes \cdots \otimes 1 
			\to \overline{f}_{(j-1)n} \phi_{a + (j-1)n},
	\end{align*}
	where on the left $\phi_a$ appears in the $j$th tensor factor of $\Cl(V \oplus V^*)^{\otimes m}$.
\end{prop}

\begin{rmk}\label[rmk]{ordinary tensor prod super alg struct}
	As the Clifford algebra is a superalgebra, clearly $\Cl(V \oplus W, \, \beta_1 \oplus \beta_2) \iso \Cl(V, \beta_1) \, \widehat{\otimes} \, \Cl(W, \beta_2)$, with $\widehat{\otimes}$ denoting the graded (or signed) tensor product of superalgebras.
	Furthermore, there is a natural superalgebra structure on $\Cl(V \oplus V^*)^{\otimes m}$ using the ordinary tensor product, but the product structure is reflecting the fact that Clifford algebras are not supercommutative.
	In particular, the volume elements $\overline{f}_r$ introduce signs when appropriate.
\end{rmk}

\section{Quantized Clifford algebras}
\label{clqnk props}

In this section we prove structural properties of the quantum Clifford algebra.
To begin, we provide a generalized definition of the quantized Clifford algebra.
The first main result in this section is an explicit basis in \Cref{thm:basis_general} proven by giving a reduction algorithm.
Our other main result is computing the center in \Cref{clqnk center}.
This computation relies on the tensor product factorization of \Cref{clqn clqnm embedding}, the quantum analog of \Cref{cl tensor m into clnm}, and the construction a quantum volume element, which enjoys many of the same properties as its classical counterpart.


\begin{dfn}\label[defn]{clqnk defn}
Choose $q \in \field^{\times}$ and let $n, k \in \ZZ_{>0}$.
The \defn{quantum Clifford algebra} $\Cl_q(n, k)$ of \defn{rank} $n$ and \defn{twist} $k$ is the  associative $\field$-algebra generated by $\psi_a$, $\psi_a^*$, $\omega_a$, and $\omega_a^{-1}$, for $a \in \{1, \dotsc, n\}$, subject to the relations
\begin{align}\label{rels defining qcl}
\omega_a \omega_b & = \omega_b \omega_a, \nonumber
& \omega_a \omega_a^{-1} & = 1, 
\\ 
\omega_a \psi_b & = q^{\delta_{ab}} \psi_b \omega_a,
& \omega_a \psi_b^* & = q^{-\delta_{ab}} \psi_b^* \omega_a,
\nonumber
\\ 
\psi_a \psi_b & + \psi_b \psi_a = 0,
& \psi_a^* \psi_b^* & + \psi_b^* \psi_a^* = 0,
\\ 
\psi_a \psi_a^* & + q^k \psi_a^* \psi_a = \omega_a^{-k},
& \psi_a \psi_a^* & + q^{-k} \psi_a^* \psi_a = \omega_a^k,
\nonumber
\\ 
\psi_a \psi_b^* & + \psi_b^* \psi_a = 0 && \text{if } a \neq b. \nonumber
\end{align}
\end{dfn}

The parameter $n$ refers to the number of \textit{pairs} of $(\psi_j, \psi_j^\dagger)$ generators, so we have only defined a quantum analog for the \textit{even} dimensional Clifford algebra $\Cln$.
The case $k = 2$ was given by Hayashi~\cite{Hayashi90} and the $k=1$ version was utilized by Kwon~\cite{Kwon14}.

The defining relations imply that
\begin{equation}
\label{eq:psiq_products}
(q \omega_a)^k - (q \omega_a)^{-k} = (q^k - q^{-k}) \psi_a \psi_a^*,
\qquad\qquad
\omega_a^{-k} - \omega_a^k = (q^k - q^{-k}) \psi_a^* \psi_a.
\end{equation}
When $q^{2k} \neq 1$, we obtain an alternative presentation of $\Cl_q(n, k)$ using \Cref{eq:psiq_products} by replacing the $q^{\pm k}$-commutor relations for $\psi_a$ and $\psi_a^*$ with
\begin{equation}
\label{eq:psi_products}
\psi_a \psi_a^* = \frac{(q \omega_a)^k - (q \omega_a)^{-k}}{q^k - q^{-k}},
\qquad\qquad
\psi_a^* \psi_a = - \frac{\omega_a^k - \omega_a^{-k}}{q^k - q^{-k}}.
\end{equation}

Notice there is a $\ZZ^n$-grading on $\Cl_q(n, k)$ defined by
\begin{equation}
\label{eq:ZZn_grading}
\deg(\psi_a) = e_a, 
\qquad\qquad
\deg(\psi_a^*) = -e_a,
\qquad\qquad
\deg(\omega_a) = 0.
\end{equation}
This naturally yields a $\ZZ$-grading (by the linear map $e_a \mapsto 1 \in \ZZ$ for all $a$) and a $\ZZ_2$-grading (by $e_a \mapsto 1 \in \ZZ_2$), so $\Cl_q(n, k)$ is a superalgebra.

\subsection{Basis}
\label{sec:basis}

We construct a series of reduction rules by manipulating the defining relations.
We then obtain bases for $\Cl_q(n, k)$ using these reduction rules.

\begin{lem}\label[lem]{reduction rules}
	The following identities hold in $\Cl_q(n, k)$.
	\begin{align}
		\psi_a \omega_a^{2k} &= q^{-2k} \psi_a, 
		\label{eq:reduction_pww} \\
		\psi_a^* \omega_a^{2k} &= \psi_a^*,
		\label{eq:reduction_dww} \\
		\omega_a^{-k} &= (q^{2k} + 1) \omega_a^k - q^{2k} \omega_a^{3k} \label{eq:reduction_vi}.
	\end{align}
\end{lem}

\begin{proof}
The defining relations imply $\psi_a^2 = (\psi_a^*)^2 = 0$. In addition, they imply
\[
\psi_a \omega_a^{-k} = \psi_a( \psi_a \psi_a^* + q^k \psi_a^* \psi_a ) = q^k \psi_a \psi_a^* \psi_a
= q^k \psi_a (q^k \omega^k - q^k \psi_a \psi_a^*)
= q^{2k} \psi_a \omega_a^k.
\]
Multiplying both sides on the right by $q^{-2k} \omega^k$ yields \Cref{eq:reduction_pww}. Similarly, we calculate that
\[
\psi_a^* \omega_a^k = \psi_a^* ( \psi_a \psi_a^* + q^{-k} \psi_a^* \psi_a ) = \psi_a^* \psi_a \psi_a^*
= \psi_a^* (\omega_a^{-k} - q^k \psi_a^* \psi_a)
= \psi_a^* \omega_a^{-k}.
\]
This implies \Cref{eq:reduction_dww}. Applying \Cref{eq:psiq_products,eq:reduction_pww} yields
\begin{align*}
\omega_a^{-k} - \omega_a^k & = (q^k - q^{-k}) \psi_a^* \psi_a 
= q^{2k} (q^k - q^{-k}) \psi_a^* \psi_a \omega_a^{2k} 
= q^{2k} (\omega_a^{-k} - \omega_a^k) \omega_a^{2k}
\\ &= q^{2k} \omega_a^k - q^{2k}\omega_a^{3k}.
\end{align*}
Solving for $\omega_a^{-k}$ yields \Cref{eq:reduction_vi}.
\end{proof}

We can use \Cref{eq:reduction_vi} to express $\omega_a^{-1}$ as a polynomial in $\field [\omega_a]$, since it is equivalent to
\begin{equation}\label{eq:reduction_v4}
	\omega_a^{4k} - (1 + q^{-2k}) \omega_a^{2k} + q^{-2k} = 0.
\end{equation}
Thus, the quantum Clifford algebra is generated by the $3n$ generators $\psi_a, \psi_a^*, \omega_a$, for $a = 1, \ldots, n$.
Notice that if $q^{2k} = 1$, \Cref{eq:reduction_v4} factors as
\begin{equation}
	\label{eq:omega_factoring_unity}
	(\omega_a^{2k} - 1)^2 = 0.
\end{equation}

\begin{lem}\label[lem]{generic reduction relations}
	The following relations hold in $\Cl_q(n, k)$:
\begin{align*}
\psi_a^* \psi_a &= q^k \omega_a^k - q^k \psi_a \psi_a^*, &
\psi_a \psi_a^* \psi_a &= q^k \psi_a \omega_a^k, \\ 
\psi_a \psi_a^* \psi_a \psi_a^* &= \psi_a \psi_a^* \omega_a^k, &
\psi_a^* \psi_a \psi_a^* &= \psi_a^* \omega_a^k, \\ 
\omega_a^{2k} &= (1  - q^{-2k}) \psi_a \psi_a^* \omega_a^k + q^{-2k}. 
\end{align*}
\end{lem}

\begin{proof}
We obtain the last relation as follows:
\[
\omega_a^{2k} = \psi_a \psi_a^* \omega_a^k + q^{-k} \psi_a^* \psi_a \omega_a^k
= \psi_a \psi_a^* \omega_a^k + q^{-k} (q^{-k} \omega_a^{-k} - q^{-k} \psi_a \psi_a^*) \omega_a^k
= (1  - q^{-2k}) \psi_a \psi_a^* \omega_a^k + q^{-2k}.
\]
The remaining identities follow immediately from the defining relations \eqref{rels defining qcl}.
\end{proof}

We construct a basis of $\Cl_q(n, k)$ by considering a normal ordering of the generators analogous to the standard basis for the classical Clifford algebra.

\begin{thm}
\label{thm:basis_general}
Consider the set $\mcB$ of monomials
\[
	\prod_{a=1}^n \psi_a^{p_a} (\psi_a^*)^{d_a} \omega_a^{v_a},
\]
where $p_a, d_a \in \{0, 1\}$ and $v_a \in \{0, 1, \dotsc, 2k-1\}$.
Then $\mcB$ is a basis of $\Cl_q(n, k)$.
\end{thm}

\begin{proof}
Clearly $\mcB$ contains $1$ and the generators of $\Cl_q(n, k)$.
Thus, we need to show that $\mcB$ is linearly independent and closed under multiplication.
Together with Lemma~\ref{generic reduction relations}, the $q$- and skew-commuting relations show that $\mcB$ is closed under multiplication.
Indeed, these are a set of rewriting rules to reduce any product of generators into a linear combination of elements in $\mcB$.
Furthermore, note that we cannot apply the rewriting rules to any element of $\mcB$ (after having chosen an order on the elements $\psi_a, \psi_a^*, \omega_a$).
Hence, we just need to show $\mcB$ is linearly independent.

To show the linear independence, we need to show that the result of the rewriting rules is unique and that they undo the application of any of the defining relations.
The latter property is clear (the only non-trivial one is the $q$-commutor relations between $\psi_a$ and $\psi_a^*$).
The uniqueness is a straightforward argument using the diamond lemma.
\end{proof}

Observe that the linear independence of the monomials in $\mcB$ does not depend on the particular order of the factors we chose in \Cref{thm:basis_general}; other orderings are possible and in fact they all differ by a power of $q$ and a sign with the exception of $\psi_a^* \psi_a$.
For the replacement $\psi_a \psi_a^* \mapsto \psi_a^* \psi_a$, we simply use the alternative reduction rule $\psi_a \psi_a^* = \omega_a^k - q^{-k} \psi_a^* \psi_a$.

We immediately obtain the dimension of $\Cl_q(n, k)$ from \Cref{thm:basis_general}.

\begin{cor}\label[cor]{clqnk dim}
We have
\[
\dim \Cl_q(n, k) = (8k)^n.
\]
\end{cor}

Now notice that when $q^{2k} = 1$, \Cref{eq:omega_factoring_unity} implies that $\omega_a^{2k} = 1$.
We can also see this from the defining relations as
\[
\psi_a \psi_a^* + q^k \psi_a^* \psi_a = \omega_a^{-k} = \omega_a^k
\]
since $q^k = q^{-k} = \pm 1$.
Conversely, when $q^{2k} \neq 1$, we can construct another basis that demonstrates we have a commutative algebra of dimension $(4k)^n$ generated by $\omega_1, \ldots, \omega_n$.

\begin{thm}
\label{thm:basis_alt}
Suppose $q^{2k} \neq 1$.
Consider the set $\mcB'$ of monomials
\[
\prod_{a=1}^n \psi_a^{p_a} (\psi_a^*)^{d_a} \omega_a^{v_a},
\]
with $p_a, d_a \in \{0, 1\}$ and $v_a \in \{0, 1, \dotsc, 4k-1\}$, satisfying the following conditions for all $a$:
\begin{enumerate}
\item $p_a + d_a < 2$,
\item $v_a \leq k$ if $p_a + d_a = 1$.
\end{enumerate}
Then $\mcB'$ is a basis of $\Cl_q(n, k)$.
\end{thm}

We provide two proofs of \Cref{thm:basis_alt}.
The first uses a series of reduction rules. It is analogous to the proof of \Cref{thm:basis_general}.
The second uses an explicit invertible change-of-basis map.

\begin{proof}[Proof: Reduction rules]
The reduction rules are given by \Cref{reduction rules}, the defining skew-commuting and $q$-commuting relations, and 
\begin{subequations}
\label{eq:replaced_qcomm_rels}
\begin{align}
\psi_a \psi_a^* & = \frac{q^k \omega_a^{3k} - q^{-k} \omega_a^k}{q^k - q^{-k}} = \frac{\omega_a^k (q^{2k} \omega_a^{2k} - 1)}{q^{2k} - 1}, \label{eq:replaced_qcomm_pd}
\\ \psi_a^* \psi_a & = \frac{q^{2k} \omega_a^k (1 - \omega_a^{2k})}{q^k - q^{-k}} = \frac{q^{3k} \omega_a^k (1 - \omega_a^{2k})}{q^{2k} - 1},
\end{align}
\end{subequations}
which are obtained from \Cref{eq:psi_products} by removing $\omega_a^{-k}$ using \Cref{eq:reduction_vi}.
The remainder of the proof is similar to the proof of \Cref{thm:basis_general}.
\end{proof}

\begin{proof}[Proof: Change-of-basis]
We will construct an invertible linear map from $\mcB$ from Theorem~\ref{thm:basis_general} to $\mcB'$.
Consider a monomial $\prod_{a=1}^n \psi_a^{p_a} (\psi_a^*)^{d_a} \omega_a^{v_a}$ in $\mcB$. It belongs to $\mcB'$ if and only if $p_a + d_a < 2$ for all $a$, so we need only consider the case when $p_a + d_a = 2$ for some $a$.
Using \Cref{eq:replaced_qcomm_pd}, we can rewrite
\[
\psi_a \psi_a^* \omega_a^{v_a} = \frac{q^{2k} \omega_a^{3k+v_a} - \omega_a^{k+v_a}}{q^{2k} - 1}.
\]
If $0 \leq v_a < k$, this is an element of $\mcB'$ with inverse image
\[
\omega_a^{3k+v_a} = q^{-2k} \bigl( (q^{2k} - 1) \psi_a \psi_a^* \omega_a^{v_a} + \omega_a^{k+v_a} \bigr) \in \mcB.
\]
Otherwise we have $k \leq v_a < 2k$, and so we may apply \Cref{eq:reduction_v4} to obtain
\begin{align*}
\psi_a \psi_a^* \omega_a^{v_a} & = \frac{q^{2k} \omega_a^{v_a-k} \bigl( (1 + q^{-2k}) \omega_a^{2k} - q^{-2k} \bigr) - \omega_a^{k+v_a}}{q^{2k} - 1}
= \frac{\omega_a^{v_a+k} - q^{-2k} \omega_a^{v_a-k}}{q^{2k} - 1} \in \mcB'.
\end{align*}
In this case, the inverse maps to
\[
\omega_a^{2k+u_a} = (q^{2k} - 1) \psi_a \psi_a^* \omega_a^{u_a+k} + q^{-2k} \omega_a^{u_a} \in \mcB,
\]
where $u_a = v_a - k$ and $0 \leq u_a < v_k$.
\end{proof}

We note that there is a commutative subalgebra $\mcC_q(n,k)$ of dimension $(4k)^n$ in $\Cl_q(n,k)$ generated by $\langle \omega_a, \psi_a \psi_a^{\dg} \mid a = 1, \dotsc, n \rangle$ regardless of whether $q^{2k} = 1$.
\Cref{thm:basis_alt} says that when $q^{2k} \neq 1$, the subalgebra $\mcC_q(n,k)$ is isomorphic to
\begin{equation}
\label{eq:commutative_quotient_nonroot_unity}
\field [\omega_1, \dotsc, \omega_n] / ( \omega_1^{4k} - (1 + q^{-2k}) \omega_1^{2k} + q^{-2k}, \dotsc, \omega_n^{4k} - (1 + q^{-2k}) \omega_n^{2k} + q^{-2k} ).
\end{equation}
Contrast this with the case $q^{2k} = 1$, where this commutative subalgebra is isomorphic to
\begin{equation}
\label{eq:commutative_quotient_root_unity}
\field [\omega_1, \eta_1, \dotsc, \omega_n, \eta_n]  / ( \omega_1^{2k} - 1, \eta_1^2 - \eta_1 \omega_1^k, \dotsc, \omega_n^{2k} - 1, \eta_n^2 - \eta_n \omega_n^k )
\end{equation}
under the map $\omega_a \mapsto \omega_a$ and $\eta_a \mapsto \psi_a \psi_a^{\dg}$.
We can compute a Gr\"obner basis for the ideal in~\eqref{eq:commutative_quotient_root_unity} with respect to degree reverse lexicographic order by taking its generators and adding $\eta_a^3 - \eta_a$ for all $a = 1, \dotsc, n$.
While not immediate, it turns out the commutative algebras are in fact isomorphic as we will see below.

The bases in \Cref{thm:basis_general} and \Cref{thm:basis_alt} have been implemented in \textsc{SageMath} by the second author.

\subsection{Structure Theorem}
\label{clqnk alg struct}

We construct a quantum analog of the isomorphism $\Gamma_V$ defined in \Cref{cl tensor m into clnm}. The (anti)commutators
\begin{equation}
\label{eq:anticommutators}
\{ \psi_a, \psi_a^{\dg} \} 
= (1 - q^k) \psi_a \psi_a^{\dg} + q^k \omega_a^k,
\qquad\qquad
[ \psi_a, \psi_a^{\dg} ] 
= (1 + q^k) \psi_a \psi_a^{\dg} - q^k \omega_a^k,
\end{equation}
will play an important role in many calculations in this section.
Clearly $\{ \psi_a, \psi_a^{\dg}\} \neq 0, 1$ for all $q \in \field^{\times}$.
We first show that the anticommutators are nontrivial central elements.
This does not have a classical analog as the classical anticommutation relations state $\{\psi_a, \psi_a^{\dg}\} = 1$.
Let $Z_q(n, k) := Z(\Cl_q(n, k))$ denote the center of $\Cl_q(n, k)$.
The next result follows from a straightforward computation, left as an exercise to the reader, using the reduction rules of \Cref{generic reduction relations}.

\begin{lem}\label[lem]{anticomm are central}
For each $a = 1, \ldots, n$, the anticommutator $\{\psi_a, \psi_a^\dg\}$ belongs to $Z_q(n, k)$ and $\{\psi_a, \psi_a^\dg\}^2 = 1$.
\end{lem}

We need quantum analogs of the volume elements defined by \Cref{classical volume elt}.
We proceed exactly as in the classical case.
For each $r = 1, \ldots, n$, the \defn{quantum volume} or \defn{chirality} element $f_r$ of the subalgebra $\Cl_q(r, k) \subseteq \Cl_q(n, k)$ is given by 
\begin{align}\label{volume elt}
	f_r = [\psi_1, \psi_1^\dg] \cdots [\psi_r, \psi_r^\dg].
\end{align}

Like in the classical case, we can express each volume element as a product of standardized coordinates:
\[
	f_r = \epsilon_1 \cdots \epsilon_{2r},
\]
with $\epsilon_j$ given by the same formulae as in \Cref{classical std coords}.
The standardized coordinates in $\Cl_q(n, k)$ are a quantum analog of a basis of $\field^{2n} \iso \field^n \oplus (\field^n)^*$ where the defining bilinear form is diagonal, but they do \textit{not} satisfy $\{\epsilon_a, \epsilon_b\} = \delta_{ab}$.
Nevertheless, the anticommutators $\{\epsilon_a, \epsilon_b\}$ are still central; this follows from an easy calculation left as an exercise to the reader.

\begin{prop}\label[prop]{std coords rels}
Let $\epsilon_1, \ldots, \epsilon_{2n}$ denote the standardized coordinates of $\Cl_q(n, k)$ defined using formulae as in \Cref{classical std coords}. Then
\[
	\{\epsilon_a, \epsilon_b \} = \epsilon_a \epsilon_b + \epsilon_b \epsilon_a 
	= 2 \delta_{ab} \{\psi_{\lceil a/2\rceil}, \psi_{\lceil a/2 \rceil}^\dg\}
\] 
is central in $\Cl_q(n, k)$.
\end{prop}

Next we prove a quantum analog of~\cite[Prop.~3.3]{michelsohn_lawson}, obtaining some relations involving $f_r$.
We note that many of the signs there vanish here because $f_r$ is the product of an \textit{even} number of factors.

\begin{prop}\label[prop]{volume elt props}
	Let $f_r$ denote a quantum volume element in $\Cl_q(n, k)$ as defined in~\eqref{volume elt}.
	Let $\phi_a$ denote either $\psi_a$ or $\psi_a^\dg$.
	\begin{itemize}
	\item If $a > r$, then $\phi_a f_r = f_r \phi_a$ and $\omega_a f_r = f_r \omega_a$.
	\item If $a \leq r$, then $\phi_a f_r = - f_r \phi_a$.
	\item For any $r, s \leq n$, we have $f_r f_s = f_s f_r$ and $f_r^2 = 1$.
	\end{itemize}
\end{prop}

\begin{proof}
The anticommutation relations imply $\phi_a$ commutes with $[\psi_b, \psi_b^\dg]$ whenever $a \neq b$.
In addition, combining \Cref{eq:anticommutators} with \Cref{generic reduction relations} implies
\begin{align*}
	[\psi_a, \psi_a^\dg] \psi_a 
	& = (1 + q^k) \psi_a \psi_a^\dg \psi_a - q^{2k} \psi_a \omega_a^k 
	= q^k \psi_a \omega_a^k 
	\\ & = -\psi_a \left((1 + q^k) \psi_a \psi_a^\dg - q^k \omega^k\right)
	= -\psi_a [\psi_a, \psi_a^\dg].
\end{align*}
A similar calculation shows that $[\psi_a, \psi_a^\dg] \psi_a^\dg = -\psi_a^\dg [\psi_a, \psi_a^\dg]$.
This means $\phi_a f_r = f_r \phi_a$ if $a > r$ and $\phi_a f_r = -f_r \phi_a$ if $a \leq r$.
The remaining relations follow from straightforward calculations.
\end{proof}

We are now ready to construct an isomorphism $\Gamma_q \colon \Cl_q(n, k)^{\otimes m} \to \Cl_q(nm, k)$, obtaining the quantum analog of \Cref{cl tensor m into clnm}.
Much like its classical counterpart $\Cln$, the quantum Clifford algebra $\Cl_q(n, k)$ is a superalgebra: it is equipped with the $\ZZ_2$-grading associated to the $\ZZ^n$-grading of~\eqref{eq:ZZn_grading}.
Therefore, the considerations in \Cref{ordinary tensor prod super alg struct} also apply here.
In particular, there is a natural isomorphism $\Cl_q(nm, k) \iso \Cl_q(n, k) \, \widehat{\otimes} \, \cdots \, \widehat{\otimes} \, \Cl_q(n, k)$.
However, using the ordinary tensor product here helps us factor quantum group actions on tensor products of modules through $\Cl_q(nm, k)$, \textit{e.g.}, as in~\cite{willie_a,willie_bd}.
As in the classical case, our isomorphism relies on the volume elements $f_r$ to compensate that we are using the ordinary tensor product instead of the superalgebra tensor product: the $f_r$ introduce signs where needed.

\begin{thm}\label{clqn clqnm embedding}
	Let $\phi_a$ denote either $\psi_a$ or $\psi_a^\dg$.
	Then the map $\Gamma_q\colon \Cl_q(n, k)^{\otimes m} \to \Cl_q(nm, k)$ defined by
\begin{align*}
	1 \otimes \cdots \otimes \phi_a \otimes \cdots \otimes 1 
		&\mapsto
	f_{(j-1)n} \phi_{a + (j-1)n}
	\text{ and} \\
	1 \otimes \cdots \otimes \omega_a \otimes \cdots \otimes 1 
		&\mapsto
	\omega_{a + (j-1)n},
\end{align*}
is an isomorphism of $\ZZ^{nm}$-graded algebras.
The inverse $\inv{\Gamma_q}\colon \Cl_q(nm, k) \to \Cl_q(n, k)^{\otimes m}$ is defined by
\begin{align*}
	\phi_{a + (j-1)n} 
		&\mapsto
	f_n \otimes \cdots \otimes f_n \otimes \phi_a \otimes 1 \otimes \cdots \otimes 1
	\text{ and} \\
	\omega_{a + (j-1)n} 
		&\mapsto
	1 \otimes \cdots \otimes \omega_a \otimes \cdots \otimes 1.
\end{align*}
In particular, 
$
	\Gamma_q(1 \otimes \cdots \otimes \phi_a \phi_b \otimes \cdots \otimes 1) 
	= 
	\phi_{a +(j-1)n} \phi_{b + (j-1)n}.
$ 
In all cases $\phi_a$, $\phi_a \phi_b$, and $\omega_a$ appear in the $j$th tensor factor of $\Cl_q(n, k)^{\otimes m}$. 
\end{thm}

\begin{proof}
Let $\phi_{a,j}$ denote the tensor product $1 \otimes \cdots \otimes \phi_a \otimes \cdots \otimes 1$ with $\phi_a$ appearing in the $j$th tensor factor. First note that $\Gamma_q$ maps each $\Cl_q(n, k)$ factor homomorphically into $\Cl_q(nm, k)$. In particular, whenever $a \neq b$, \Cref{volume elt props} implies 
\begin{align*}
	\Gamma_q(\phi_{a, j})\Gamma_q(\phi_{b, j})
		&= f_{(j-1)n}^2 \phi_{a + (j-1)n}\phi_{b + (j-1)n} 
		= -\Gamma_q(\phi_{b, j})\Gamma_q(\phi_{a, j}).
\end{align*}
In addition, for any $a, b$, \Cref{volume elt props} implies
\begin{align*}
	\Gamma_q(\omega_{a, j}) \Gamma_q(\psi_{b,j}) \inv{\Gamma_q(\omega_{a, j})} 
		&= f_{(j-1)n} \omega_{a + (j-1)n} \psi_{b + (j-1)n} \inv{\omega_{a + (j-1)n}}
		= q^{\delta_{ab}} \Gamma_q(\psi_{b,j}),
\end{align*}
and
\begin{align*}
	\Gamma_q(\psi_{a,j}) \Gamma_q(\psi_{a,j}^\dg) + q^{\pm k} \Gamma_q(\psi_{a,j}^\dg) \Gamma_q(\psi_{a,j}) 
		= \big\{\psi_{a + (j-1)n}, \psi_{a + (j-1)n}^\dg\big\}_{q^{\pm k}} 
		= w_{a + (j-1)n}^{\mp k} 
		= \Gamma_q(\omega_{a,j})^{\mp k}.  
\end{align*}

Next notice $\Gamma_q$ indeed defines an algebra map on $\Cl_q(n, k)^{\otimes m}$. Without loss of generality, suppose that $j < k$. \Cref{volume elt props} implies $f_{(j-1)n}$ commutes with $\phi_{b + (k-1)n}$ and $\omega_{b + (k-1)n}$. Thus for any $a, b$,
\begin{align*}
	\Gamma_q(\phi_{a,j}) \Gamma_q(\phi_{b, k})
		&= -f_{(j-1)n} f_{(k-1)n} \phi_{a + (j-1)n} \phi_{b + (k-1)n} 
		= \Gamma_q(\phi_{b, k}) \Gamma_q(\phi_{a, k})
		\text{  and} \\[+0.5em]
		\Gamma_q(\phi_{a,j}) \Gamma_q(\omega_{b, k})
		&= \omega_{b + (k-1)n} f_{(j-1)n} \phi_{a + (j-1)n} 
		= \Gamma_q(\omega_{b, k}) \Gamma_q(\phi_{a,j}).
\end{align*}

We conclude by showing $\Gamma_q$ is an isomorphism with inverse $\Gamma_q^{-1}$. It suffices to show that $\Gamma_q \circ \Gamma_q^{-1}$ and $\Gamma_q^{-1} \circ \Gamma_q$ coincide with appropriate identity maps on a set of generators. To this end, notice that if the commutator appears in the $j$th factor of $1 \otimes \cdots \otimes [\psi_a^{\phdg}, \psi_a^\dg] \otimes \cdots \otimes 1$, then 
\begin{align*}
	\Gamma_q(1 \otimes \cdots \otimes [\psi_a^{\phdg}, \psi_a^\dg] \otimes \cdots \otimes 1) 
		= f_{(j-1)n}^2 [\psi_{a + (j-1)n}^{\phdg}, \psi_{a + (j-1)n}^\dg] 
		= \big[\psi_{a + (j-1)n}^{\phdg}, \psi_{a + (j-1)n}^\dg\big]
\end{align*}
because $\Gamma_q$ is an algebra map and $f_r^2 = 1$ by \Cref{volume elt props}. Since $f_n = [\psi_1, \psi_1^\dg] \cdots [\psi_n, \psi_n^\dg]$, we see that
\begin{align*}
	\Gamma_q \circ \Gamma_q^{-1}(\phi_{a + (j-1)n}) 
	&= \Gamma_q(f_n \otimes \cdots \otimes f_n \otimes \phi_a \otimes 1 \otimes \cdots \otimes 1) \\
	&= \prod_{k=1}^{j-1} \left(\prod_{a=1}^n \left[\psi_{a + (j-1)n}^{\phdg}, \psi_{a + (j-1)n}^\dg\right]\right) 
			\cdot f_{(j-1)n} \phi_{a + (j-1)n} \\
	&= f_{(j-1)n}^2 \phi_{a + (j-1)n} \\
	&= \phi_{a + (j-1)n},
\end{align*}
and similarly $\Gamma_q \circ \Gamma_q^{-1}(\omega_{a + (j-1)n}) = \Gamma_q(\omega_{a,j}) = \omega_{a + (j-1)n}$. Hence $\Gamma_q^{-1}$ is in fact a right inverse of $\Gamma_q$. Similar calculations show $\Gamma_q^{-1}$ also inverts $\Gamma_q$ on the left.
\end{proof}

The following structure result is a special case of \Cref{clqn clqnm embedding}.

\begin{cor}\label[cor]{clqn is n fold product of cl1}
	There is an isomorphism of $\ZZ^n$-graded algebras 
	\[
	\Cl_q(n, k) \cong \Cl_q(1, k)^{\otimes n}.
	\]
\end{cor}

\subsection{The center of a quantum Clifford algebra}
\label{sec:clqnk center}

As we saw in \Cref{anticomm are central}, the anticommutators $\{\psi_a, \psi_a^{\dg}\}$ are nontrivial central elements that are square roots of unity.
We now construct their $k$th roots in $\Cl_q(n, k)$ and show that these roots generate the center $Z_q(n, k)$ of $\Cl_q(n, k)$.

\begin{thm}\label{clqnk center}
	We have
	\[
	 	Z_q(n, k) 
	 	\iso
	 	\field [\ZZ_{2k}] \otimes \cdots \otimes \field [\ZZ_{2k}]
	 	=
	 	\field [\ZZ_{2k}^n].
	\]
	In particular, $Z_q(n, k)$ is generated by the order $2k$ elements
	\begin{equation}
	\label{eq:center_gen}
	z_a = q \omega_a - (q-1) \psi_a^{\phdg} \psi_a^\dg \omega_a^{k+1},
	\end{equation}
	for $a = 1, \ldots, n$.
	Moreover, $\dim Z_q(n,k) = (2k)^n$ and each $z_a$ satisfies
	\begin{equation}\label{anticomm zj rel}
		\psi_a^{\phdg} \psi_a^\dg + \psi_a^\dg \psi_a^{\phdg} = z_a^k.
	\end{equation}
\end{thm}

We first show that the elements $z_a$ are in the center, and we write their powers explicitly in terms of our basis from \Cref{thm:basis_general}.

\begin{lem}\label[lem]{zkr formula}
	The elements $z_a$ defined in \Cref{clqnk center} are central in $\Cl_q(n, k)$ and they satisfy
\begin{align}\label{zjr}
	z_a^r = q^r \omega_a^r - (q^r - 1) \psi_a^{\phdg} \psi_a^\dg \omega_a^{k + r}.
\end{align}
In particular, $z_a^k = \{\psi_a, \psi_a^\dg\}$, so $z_a^{2k} = 1$.
\end{lem}

\begin{proof}
	That $z_a$ is central is a straightforward computation showing it commutes with all of the generators of $\Cl_q(n, k)$ using the reduction rules.
	\Cref{zjr} is a straightforward calculation using induction on $r$.
	That $z_a^k = \{\psi_a, \psi_a^{\dg}\}$ is immediate from comparing \Cref{eq:anticommutators} with \Cref{zjr} with $r = k$.
	The last claim is \Cref{anticomm are central}.
\end{proof}

\begin{proof}[Proof of \Cref{clqnk center}.]
	Using the decomposition of \Cref{clqn is n fold product of cl1}, it is sufficient to show the claim for $\Cl_q(1, k)$; that is, $Z_q(1, k) \iso \field [ \ZZ_{2k} ]$.
	
	We will show the central elements $z_1, \ldots, z_1^{2k}$ constitute a basis of $Z_q(1, k)$.
	The elements $z_1, \ldots, z_1^{2k}$ are linearly independent because the elements
	\[
	1, \omega_1, \ldots, \omega_1^{2k-1}, \psi_1 \psi_1^\dg, \psi_1 \psi_1^\dg \omega_1, \ldots, \psi_1 \psi_1^\dg \omega_1^{2k-1}
	\]
	are linearly independent by \Cref{thm:basis_general}.
	Hence the subalgebra generated by $z_1$ is central and isomorphic to $\field [ \ZZ_{2k} ]$ by \Cref{zkr formula}.
	To show there is nothing else in the center, consider an arbitrary central element $\zeta \in Z_q(1, k)$.
	In what follows we omit the subscript $1$ on $\Cl_q(1, k)$ generators for brevity.
	Using \Cref{thm:basis_general}, we write
	\[
	\zeta = \sum_{j=0}^{2k-1} a_j \omega^j + \psi \sum_{j=0}^{2k-1} b_j \omega^j + \psi^\dg \sum_{j=0}^{2k-1} c_j \omega^j + \psi \psi^\dg \sum_{j=0}^{2k-1} d_j \omega^j
	\]
	for some constants $a_j, b_j, c_j, d_j \in \field$.
	By direct computation using the reduction rules, we have
	\begin{align*}
	0 = [\zeta, \psi] & = \psi \sum_{j=0}^{2k-1} \left( (q^j - 1) a_j + q^j d_{j+k} \right) \omega^j + \sum_{j=0}^{2k-1} q^{j+k} c_j \omega^j - \psi \psi^{\dg} \sum_{j=0}^{2k-1} (q^{j+k} + 1) c_j \omega^j,
	\\
	0 = [\zeta, \psi^{\dg}] & = \psi^* \sum_{j=0}^{2k-1} \left( (q^{-j} - 1) a_j + d_{j+k} \right) \omega^j - q^k \sum_{j=0}^{2k-1} b_j \omega^{j+k} + q^k \psi \psi^* \sum_{j=0}^{2k-1} (q^{-j-k} + 1) b_j \omega^j,
	\\
	\sum_{j=0}^{2k-1} b_j \omega^{j+k}
	& = \sum_{j=0}^{k-1} (b_j \omega^k + b_{j+k} q^{-2k}) \omega^j + (1 - q^{-2k}) \sum_{j=0}^{k-1} b_{j+k} \psi \psi^{\dg} \omega^{j+k},
	\end{align*}
	where we take the indices modulo $2k$.
	Therefore, by \Cref{reduction rules} and linear independence from \Cref{thm:basis_general}, we must have $b_j = c_j = 0$ for all $j$.
	Furthermore, we must have $d_j = (q^{-j} - 1) a_{k + j}$.
	This means there are at most $2k$ degrees of freedom, which yields $\dim\left(Z_q(1, k)\right) \leq 2k$ and the claim follows.
\end{proof}

We comment on two consequences of \Cref{clqnk center}.
The first is revealed by \Cref{anticomm zj rel}.
We may view $\Cl_q(n, k)$ as a classical Clifford algebra $\Cl(R^n \oplus (R^*)^n)$, albeit with a slightly rescaled pairing, over the group ring $R = \field [z_1, \ldots, z_n] \iso \field [\ZZ_{2k}^n]$, once we relax the vector space requirement in \Cref{classical clifford alg defn} to admit more general (free) modules over commutative rings.
This interpretation elucidates the conclusion of \Cref{clqnk center}: the classical Clifford algebra is simple, so the only central elements in $\Cl_q(n, k)$ can be the scalars $\field [\ZZ_{2k}^n]$ in the base ring.
This also yields that $\mcC_q(n,k)$ (see~\eqref{eq:commutative_quotient_nonroot_unity} and~\eqref{eq:commutative_quotient_root_unity}) are isomorphic for any value of $q \in \field^{\times}$ as $\mcC_q(n, k) \iso R \oplus R \langle \psi_1 \psi_1^{\dg}, \dotsc, \psi_n \psi_n^{\dg} \rangle$.
The second consequence is revealed by \Cref{eq:center_gen}: the center belongs to the degree $0$ component of $\Cl_q(n, k)$ under the $\ZZ^n$-grading given in \Cref{eq:ZZn_grading}.

\subsection{Quantum Clifford (anti-)involutions}
\label{gradings filtrations involutions}

We define some (anti-)involutions on the quantum Clifford algebra $\Cl_q(n, k)$ that are quantum analogs of (anti-)involutions of the classical Clifford algebra.
For example, we define quantum analogs of the transpose and conjugation maps.

First recall $\Cl_q(n, k)$ is a superalgebra, and so there is the \defn{grade involution} $\alpha$ on $\Cl_q(n, k)$ defined by
\[
\psi_a \mapsto -\psi_a,
\qquad\qquad
\psi_a^* \mapsto -\psi_a^*,
\qquad\qquad
\omega_a \mapsto \omega_a,
\qquad\qquad
q \mapsto q.
\]
When the twist $k$ is even, we could also define a variant $\widetilde{\alpha}$ by having $\omega_a \mapsto -\omega_a$ instead.

Next, we define the \defn{transpose} anti-involution on $\Cl_q(n, k)$ by
\[
\psi_a \mapsto \psi_a,
\qquad\qquad
\psi_a^* \mapsto \psi_a^*,
\qquad\qquad
\omega_a \mapsto q^{-1} \omega_a^{-1},
\qquad\qquad
q \mapsto q.
\]
This is the quantum analog of the transpose on the classical Clifford algebra, and we denote it by $x^t$.

\begin{prop}
The transpose is an anti-involution on $\Cl_q(n,k)$.
Moreover, the transpose commutes with the grade involution.
\end{prop}

\begin{proof}
A direct computation shows the transpose preserves the defining relations and that both $(x^t)^t = x$ and $\alpha(x^t) = \alpha(x)^t$ for all generators $x \in \Cl_q(n,k)$.
\end{proof}


With the transpose and grade involution in hand, we may define the \defn{Clifford conjugation} on $\Cl_q(n, k)$:
for any $x \in \Cl_q(n,k)$, we set
\[
	\overline{x} = \alpha(x^t).
\]

Next, we separate the anti-involution $\ast$ defined in~\cite[Sec.~5]{Hayashi90} depending on the value $q$ into the product of two separate anti-involutions that do not depend on $q$.
Hayashi used the map $\ast$ to make $\Cl_q(n, k)$ into a $\ast$-algebra and to prove the unitarity of his $\Cl_q(n, k)$-modules~\cite{Hayashi90}.
The first is the anti-involution that we call the \defn{dagger} map, which we define by
\[
\psi_a \mapsto \psi_a^*,
\qquad\qquad
\psi_a^* \mapsto \psi_a,
\qquad\qquad
\omega_a \mapsto \omega_a,
\qquad\qquad
q \mapsto q,
\]
and denote by $x^{\dagger}$ for $x \in \Cl_q(n,k)$.
The second anti-involution is the \defn{duality} map on $\Cl_q(n, k)$, which we define by
\[
\psi_a \mapsto \psi_a^*,
\qquad\qquad
\psi_a^* \mapsto \psi_a,
\qquad\qquad
\omega_a \mapsto \omega_a^{-1},
\qquad\qquad
q \mapsto q^{-1}
\]
and denote by $x^{\vee}$ for $x \in \Cl_q(n,k)$.
Both the dagger and duality maps are different quantum analogs of the natural duality anti-involution $\ast$ on $\Cln$ sending $\psi_a^* \leftrightarrow \psi_a$, although arguably the dagger map is the most natural analog.
We justify calling the latter map the duality map by the relationship with the quantum group discussed below in \Cref{q gp morphisms}.

\begin{prop}
The grade involution, transpose, dagger, and duality maps all commute.
\end{prop}

\begin{proof}
This is a direct computation on the generators, such as $x^{\vee\dagger} = x^{\dagger\vee}$ for any generator $x \in \Cl_q(n,k)$.
\end{proof}

Now by composing the duality and dagger maps,we obtain the involution of $\Cl_q(n,k)$ given by
\[
\psi_a \mapsto \psi_a,
\qquad\qquad
\psi_a^* \mapsto \psi_a^*,
\qquad\qquad
\omega_a \mapsto \omega_a^{-1},
\qquad\qquad
q \mapsto q^{-1}.
\]
This could be said to not have a classical analog as it acts as the identity map on $\psi_a$ and $\psi_a^{\dg}$.
The composition of the transpose and duality maps is the involution defined by
\[
\psi_a \mapsto \psi_a^*,
\qquad\qquad
\psi_a^* \mapsto \psi_a,
\qquad\qquad
\omega_a \mapsto q \omega_a,
\qquad\qquad
q \mapsto q^{-1}.
\]
If $q^{2k} = 1$, then we can construct another involution $\check{\kappa}$ by
\[
\psi_a \mapsto \psi_a^*,
\qquad\qquad
\psi_a^* \mapsto \psi_a,
\qquad\qquad
\omega_a \mapsto q^{-1} \omega_a,
\qquad\qquad
q \mapsto q^{-1}.
\]
These arise from the isomorphism of $\Cln$ induced from $\field^n \oplus (\field^n)^* \iso (\field^n)^* \oplus \field^n$.

\subsection{Quantum group morphisms}
\label{q gp morphisms}

We provide the twist $k$ version of~\cite[Thm.~3.2]{Hayashi90} through a similar direct computation.

Let $U_q(\lieg, k)$ denote a \textit{twisted} form of the Drinfel'd--Jimbo~\cite{Drinfeld85,Jimbo85} quantum group, defined as follows.
Let $\lieg$ denote a semisimple Lie algebra and let $A$ denote its generalized Cartan matrix with the diagonal matrix of root lengths $D = \mathrm{diag}(d_1, \ldots, d_r)$ such that $DA$ is symmetric.
Let $q \in \field$ be transcendental and write $q_i = q^{d_i}$.
Then $U_q(\lieg, k)$ is the  associative $\field$-algebra generated by the elements $E_i, F_i, K_i$, and $\inv{K_i}$, for $i = 1, \ldots, r$, subject to the relations:
\begin{gather}
K_i K_j = K_j K_i, \quad K_i \inv{K_i} = \inv{K_i} K_i = 1, \nonumber\\
K_i E_j \inv{K_i} = q_i^{a_{ij}} E_j, \quad K_i 
F_j \inv{K_i} = q_i^{-a_{ij}} F_j,  \nonumber \\
E_i F_j - F_j E_i = \delta_{ij} \frac{K_i^k - K_i^{-k}}{q_i^k - q_i^{-k}}, \label{uq rels}
\end{gather}
together with the $q$-Serre relations for $i \neq j$,
\begin{equation}\label{uq Serre rels}
	\sum_{m=0}^{1-a_{ij}} (-1)^m
	\begin{bmatrix}
	1 - a_{ij} \\ m
	\end{bmatrix}_{q_i^k}
	X_i^k X_j X_i^{1-a_{ij}-m} = 0.
\end{equation}
In the last equality every $X$ is either an $E$ or an $F$.
The $q$-integer is given by $[n]_{q_i} = \frac{q_i^n - q_i^{-n}}{q_i - \inv{q_i}}$, the $q$-factorial is $[n]_{q_i}! = [n]_{q_i}\cdots [1]_{q_i}$, and the $q$-binomial coefficient is defined analogously.
The following identity is often useful for verifying $q$-Serre relations when $a_{ij} = -1$:
\begin{align}\label{alt uq Serre rels}
\begin{split}
	\sum_{m=0}^{1-a_{ij}} (-1)^m
	\begin{bmatrix}
	1 - a_{ij} \\ m
	\end{bmatrix}_{q_i^k} 
	X_i^k X_j X_i^{1-a_{ij}-m} 
		&= X_i (X_i X_j - q_i^k X_j X_i) - q_i^{-k} (X_i X_j - q_i^k X_j X_i) X_i \\
		&= \bigl[X_i, [X_i, X_j]_{q_i^k} \bigr]_{q_i^{-k}}.
\end{split}
\end{align} 

Note that $U_q(\lieg, k)$ differs from the standard $\Uqg$ by the exponents on $K_i$ in \Cref{uq rels} and the exponents on $q_i$ in \Cref{uq rels,uq Serre rels}.
In particular, we recover the standard presentation when $k = 1$.
For more information on quantum groups, we refer the reader to one of the standard textbooks such as~\cite{chari_pressley_1994,Kassel95,Jantzen96}.

We need the following twist $k$ versions of~\cite[Lemma~3.1]{Hayashi90}.

\begin{lem}\label[lem]{identities in clq}
Suppose $q^{2k} \neq 1$. Let $a, b, c$ be distinct elements of $\{1, \ldots, n\}$ and let $\phi_a$ denote either $\psi_a$ or $\psi_a^\dg$. The following identities hold in $\Cl_q(n, k)$:
\begin{align}
	[\psi_a \psi_b^\dg, \psi_b \psi_a^\dg] &= \frac{(\omega_a \omega_b^{-1})^k - (\omega_a \omega_b^{-1})^{-k}}{q^k - q^{-k}}, \label{psi psid psi psid comm}\\
	[\psi_a \psi_b, \psi_b^\dg \psi_a^\dg] &= \frac{(q\omega_a \omega_b)^k - (q\omega_a \omega_b)^{-k}}{q^k - q^{-k}}, \label{psi psi psid psid comm}\\
	[\phi_a \psi_b, \psi_b^\dg \phi_c]_{q^{\pm k}}
		&= \omega_b^{\mp k} \phi_a \phi_c, \label{q comm 4 psi} \\
	\psi_a \psi_a^\dg \pm \psi_a^\dg \psi_a 
		&= \frac{q^k \omega_a^k \pm \omega_a^{-k}}{q^k \pm 1} \label{psi psid comm anticomm}.  
\end{align}
\end{lem}

\begin{proof}
Using the defining relations, we see that
\[
[\psi_a \psi_b^\dg, \psi_b \psi_a^\dg] = (\psi_a \psi_a^\dg) (\psi_b^\dg \psi_b) - (\psi_a^\dg \psi_a) (\psi_b \psi_b^\dg).
\]
Applying \Cref{eq:psi_products} then yields \Cref{psi psid psi psid comm}:
\begin{align*}
	[\psi_a \psi_b^\dg, \psi_b \psi_a^\dg] 
		&= \frac{-1}{(q^k - q^{-k})^2}
		\big(\bigl(q\omega_a)^k - (q\omega_a)^{-k}\bigr)
		(\omega_b^k - \omega_b^{-k}) - (\omega_a- \omega_a^{-1})
		\bigl(q\omega_b)^k - (q\omega_b)^{-k}\bigr)\bigr) \\
		&= \frac{(\omega_a \omega_b^{-1})^k - (\omega_a \omega_b^{-1})^{-k}}{q^k - q^{-k}}.
\end{align*}
\Cref{psi psi psid psid comm} follows from a similar calculation. To establish \Cref{psi psid comm anticomm}, we use \Cref{eq:psi_products}:
\begin{align*}
	\psi_a \psi_a^\dg \pm \psi_a^\dg \psi_a
		&= \frac{1}{q^k - q^{-k}}\big((q\omega_a)^k - (q\omega_a)^{-k}\bigr) \mp (\omega_a^k - \omega_a^{-k})\big) \\
		&= \frac{q^k \mp 1}{q^{2k} - 1}\big(q^k \omega_a
			\pm \omega_a^{-k}\big) \\
		&= \frac{q^k \omega_a^k \pm \omega_a^{-k}}{q^k \pm 1}.
\end{align*}
\Cref{q comm 4 psi} is similarly easy and we omit its proof.
\end{proof}

If $k$ is even or $q$ has a square root in $\field$, we may write \Cref{psi psid comm anticomm} as the more symmetric
\begin{align}\label{psi psid comm anticomm equiv}
	\psi_a \psi_a^\dg \pm \psi_a^\dg \psi_a
	&= \frac{(q^{\frac{1}{2}} \omega_a)^k \pm (q^{\frac{1}{2}} \omega_a)^{-k}}{q^{\frac{k}{2}} \pm q^{-\frac{k}{2}}}.
\end{align}

\begin{prop}\label[prop]{twisted qgp into clqnk}
	Suppose $q^{2k} \neq 1$.
	There is an algebra homomorphism $\Theta \colon U_q(\lieg, k) \to \Cl_q(n, k)$ defined on generators as follows.
	\begin{enumerate}
		\item If $\lieg = \mathfrak{sl}_n$ then 
		\begin{align*}
			E_i &\mapsto \psi_i^{\phdg} \psi_{i+1}^{\dg}, 
			&
			F_i &\mapsto \psi_{i+1}^{\phdg} \psi_i^\dg,
			&
			K_i &\mapsto \omega_i^{\phdg} \omega_{i+1}^{-1},
			& 
			&i = 1, \ldots, n-1.
		\end{align*}
		\item If $\lieg = \mathfrak{so}_{2n}$ then
		\begin{align*}
			E_i &\mapsto \psi_i^{\phdg} \psi_{i+1}^{\dg}, 
			&
			F_i &\mapsto \psi_{i+1}^{\phdg} \psi_i^\dg,
			&
			K_i &\mapsto \omega_i^{\phdg} \omega_{i+1}^{-1},
			& 
			&i = 1, \ldots, n-1, \\
			E_n &\mapsto \psi_{n-1} \psi_n,
			&
			F_n &\mapsto \psi_n^\dg \psi_{n-1}^\dg, 
			& 
			K_n &\mapsto q \omega_{n-1} \omega_n. 
		\end{align*}
		\item If $\lieg = \mathfrak{so}_{2n+1}$ and $q^{\frac{1}{2}} \in \field$, there is an algebra homomorphism $U_{q^{\frac{1}{2}}}(\lieg, k) \mapsto \Cl_q(n, k)$ defined on generators by
		\begin{align*}
			E_i &\mapsto \psi_i^{\phdg} \psi_{i+1}^{\dg}, 
			&
			F_i &\mapsto \psi_{i+1}^{\phdg} \psi_i^\dg,
			&
			K_i &\mapsto \omega_i^{\phdg} \omega_{i+1}^{-1},
			& 
			&i = 1, \ldots, n-1, \\
			E_n &\mapsto \psi_n,
			&
			F_n &\mapsto \psi_n^\dg,
			&
			K_n &\mapsto  q^{\frac{1}{2}} \omega_n. 
		\end{align*}
\end{enumerate}
\end{prop}

\begin{proof}
	The claim follows from direct calculation. We must verify that the images $\widetilde{E}_i, \widetilde{F}_i$, and $\widetilde{K}_i$ of $E_i, F_i$, and $K_i$ satisfy the relations defining $U_q(\lieg, k)$. For starters, if $i < n$ then 
	\begin{align*}
		[\widetilde{E}_i, \widetilde{F}_i] 
		= [\psi_{i}^{\phdg}\psi_{i+1}^\dg, \psi_{i+1}^{\phdg} \psi_{i}^\dg] 
		= \frac{(\omega_i \omega_{i+1}^{-1})^k - (\omega_i \omega_{i+1}^{-1})^{-t}}{q^k - q^{-k}} 
		= \frac{\widetilde{K}_i^k - \widetilde{K}_i^{-k}}{q^k - q^{-k}}.
	\end{align*}
	If $i, j < n$ and $\abs{i - j} > 1$, the canonical anticommutation relations imply $[E_i, E_j] = 0$. When $\abs{i - j} = 1$, each summand in the $q$-Serre \cref{uq Serre rels} vanishes. Since $\widetilde{F}_i = \psi_{i+1}^{\phdg} \psi_i^\dg = \widetilde{E}_i^\vee$, this means the $\widetilde{F}_i$ also satisfy the $q$-Serre relations.
	
	To prove the theorem, it remains to check that $\widetilde{E}_i, \widetilde{F}_i$, and $\widetilde{K}_i$ satisfy the appropriate relations when $i = n$. For instance, \Cref{psi psid comm anticomm equiv} implies
	\[
		[\widetilde{E}_n, \widetilde{F}_n] =  [\psi_n, \psi_n^\dg] 
		= \frac{(q^{\frac{1}{2}} \omega_n)^k - (q^{\frac{1}{2}} \omega_n)^{-k}}{q^{\frac{k}{2}} - q^{-\frac{k}{2}}}
		= \frac{\widetilde{K}_i^k - \widetilde{K}_i^{-k}}{q^{\frac{k}{2}} - q^{-\frac{k}{2}}}.
	\]
	We leave the remaining verifications as an exercise to the reader.
\end{proof}

The formulas in \Cref{twisted qgp into clqnk} capture an underlying correspondence between the simple roots of $\lieg$ and $\Cl_q(n, k)$.
Using the grading from \Cref{eq:ZZn_grading}, notice that
\begin{equation*}
	\deg(\Theta(E_i)) = e_i - e_{i+1} 
	\quad \text{and} \quad
	\deg(\Theta(F_i)) = -(e_i - e_{i+1})
\end{equation*}
whenever $i < n$. Moreover, if $\lieg = \so_{2n}$ we have
\begin{equation*}
	\deg(\Theta(E_n)) = e_{n-1} + e_n
	\quad \text{and} \quad
	\deg(\Theta(F_n)) = -(e_{n-1} + e_n), 
\end{equation*}
and similarly if $\lieg = \so_{2n+1}$ then
\begin{equation*}
	\deg(\Theta(E_n)) = e_n
	\quad \text{and} \quad
	\deg(\Theta(F_n)) = -e_n.
\end{equation*}

To conclude, we compute the maps induced by the various (anti-)involutions defined in \Cref{gradings filtrations involutions}. These become algebra and coalgebra (anti-)automorphisms on the quantum group image. To begin, note that the dagger maps acts on $\Theta\left(U_q(\lieg, k)\right)$ as the anti-automorphism~\cite[Eq.~(5.6)]{Hayashi90}
\[
E_i^{\dagger} = F_i,
\qquad\qquad
F_i^{\dagger} = E_i,
\qquad\qquad
K_i^{\dagger} = K_i,
\qquad\qquad
q \mapsto q.
\]
This map becomes a coalgebra automorphism of the quantum group as indicated in~\cite[Sec.~5.3]{Hayashi90}.
Next observe that the duality map induces the dual representation: the action on the quantum group image in $\Cl_q(n,k)$ is~\cite[Eq.~(5.8)]{Hayashi90}
\[
E_i^{\vee} = F_i,
\qquad\qquad
F_i^{\vee} = E_i,
\qquad\qquad
K_i^{\vee} = K_i^{-1},
\qquad\qquad
q \mapsto q^{-1}.
\]
As noted by~\cite[Sec.~5.3]{Hayashi90}, this is an anti-(co)algebra automorphism of the quantum group, which agrees with the fact that we need to also reverse the order of a tensor product of representations.
Finally, observe that the transpose on $\Cl_q(n, k)$ becomes the star (anti-)involution map (see, \textit{e.g.},~\cite{Lusztig90,Kashiwara91,Kashiwara93}) on the quantum group image:
\[
E_i^t = E_i,
\qquad\qquad
F_i^t = F_i,
\qquad\qquad
K_i^t = K_i^{-1},
\qquad\qquad
q \mapsto q,
\]
which is important in Kashiwara's grand loop argument for crystal bases~\cite{Kashiwara91}.

\section{Representation theory}
\label{clqnk rep theory}

In this section we discuss the representation-theoretic aspects of $\Cl_q(n, k)$.
In particular, \Cref{clqnk repn} defines a collection of irreducible $\Cl_q(n, k)$-actions on the braided exterior algebra $\qbigwedge(\field^n)$, which is defined below, and \Cref{generalized semisimplicity} proves that $\Cl_q(n, k)$ is semisimple exactly when its center $Z_q(n, k)$ is semisimple.
We note that only one of the $\Cl_q(n, k)$-modules has a classical analog.
In addition, we discuss a representation-theoretic implication of the quantum volume elements defined in \Cref{volume elt}.

Following~\cite{berenstein}, the \defn{braided exterior algebra} $\qbigwedge(V)$ is the associative algebra generated by a basis $\{ v_1, \ldots, v_n\}$ of $V$ subject to the following relations:
\begin{subequations}
\begin{align}
	v_j^2 &= 0, & j & = 1, \ldots, n \\
	v_j v_k &= -q v_k v_j, & j & < k. \label{eq:braided_q_commute}
\end{align}
\end{subequations}
The elements
\begin{align}\label{arb braided ext alg basis}
	v(\ell) = \begin{cases}
		v_1^{\ell_1} v_2^{\ell_2} \cdots v_n^{\ell_n} & \text{if } \ell_j \in \{0, 1\}, \\
		0 & \text{otherwise},
	\end{cases}
\end{align}
for $\ell \in \{0, 1\}^n$, form a basis of $\qbigwedge(V)$.
Both the classical and braided exterior algebras can be understood as quotients of a tensor algebra modulo an ideal generated by ``symmetric'' $2$-tensors, which are ``positive'' eigenvectors of appropriate braiding operators.
In the classical case, the braiding operators are trivial flip maps that transpose tensor factors.
In the quantum case, braiding operators are induced by the universal $R$-matrix of some quantum group \cite[Def.~2.1]{berenstein}.

Recall that unlike its classical counterpart $\Cln$, the quantum Clifford algebra $\Cl_q(n, k)$ has a non-trivial center, which is described in \Cref{clqnk center}.
This means there is some flexibility in defining an action of $\Cl_q(n, k)$ on $\qbigwedge(\field^n)$ depending on which roots of unity are in $\field$.
There are $n$ linearly independent central elements $z_a$ such that $z_a^{2k} = 1$, so the representations are naturally indexed by $\ZZ_{2k}^n$.
In other words, there is a unique irreducible representation for each element in the discrete subgroup $\ZZ_{2k}^n \subset Z_q(n, k)$.
Recall that $\zeta_r$ denotes a primitive $r$-th root of unity.

\begin{prop}\label[prop]{clqnk repn}
	Suppose $\zeta_{2k} \in \field$.
	For each $p = (p_1, \ldots, p_n)$ in $\ZZ_{2k}^n$, there is an irreducible $\Cl_q(n, k)$ representation $\pi_p$ on $\qbigwedge(\field^n)$ defined by
	\begin{subequations}
	\label{clqnk action}
	\begin{align}
		\psi_a \rhd v(\ell) 
			&= (-1)^{\ell_1 + \cdots \ell_{a-1}} v(\ell + e_a), \\
		\psi_a^\dg \rhd v(\ell) 
			&= (-1)^{\ell_1 + \cdots \ell_{a-1}} \cdot \zeta_{2k}^{p_a} \cdot v(\ell - e_a), \\
		\omega_a \rhd v(\ell) 
			&= \zeta_{2k}^{p_a} q^{\ell_a - 1} v(\ell),
	\end{align}
	\end{subequations}
	where $v(\ell \pm e_a) = 0$ if $\ell \pm e_a \notin \{0, 1\}^n$.
	Moreover, $\pi_p$ and $\pi_{p'}$ are not isomorphic whenever $p \neq p'$.
\end{prop}

\begin{proof}
	An easy calculation shows that $\pi_p$ indeed defines an action of $\Cl_q(n, k)$. Each $\pi_p$ is irreducible because we can obtain any basis vector $v(\ell)$ from any other by first applying a sequence of lowering operators to obtain $v(0)$ (up to a scalar), and then using an appropriate sequence of raising operators.
	Using Relations~\eqref{clqnk action}, we see that
	\begin{align*}
		\pi_p(z_a) v(\ell) 
			= \zeta_{2k}^{p_a} \left(q^{\ell_a} - (q-1)\ell_a\right) v(\ell)
			= \zeta_{2k}^{p_a} v(\ell),
	\end{align*}
	where $z_a$ is the central element defined in \Cref{clqnk center}.
	Hence $\pi_p$ and $\pi_{p'}$ are distinct representations whenever $p \neq p'$.
\end{proof}

\begin{rmk}
	We have defined the $\Cl_q(n,k)$-action directly in terms of basis vectors without using \Cref{eq:braided_q_commute}.
	As such, there is no difference between defining this on $\qbigwedge(V)$, $\bigwedge(V)$, or $\field^{2n}$, as they are all isomorphic as $\field$-modules.
	Hence, the formulae defining the $\Cl_q(n, 2)$-module $\pi_0$ have appeared previously in~\cite[Prop.~2.1]{Hayashi90}, where the representation was defined on the usual exterior algebra $\bigwedge(V)$.
	Yet, the braided exterior algebra is better behaved for taking tensor products of quantum group representations; see, \textit{e.g.},~\cite{willie_a,willie_bd}.
\end{rmk}

\begin{rmk}
\label[rmk]{rmk:dual_spinor_repr}
The spinor representation considered in~\cite{Kwon14} is a ``dual'' version of the spinor representation considered here: it is obtained by interchanging $\ell_a \leftrightarrow 1 - \ell_a$ for all $a$.
\end{rmk}

Amongst the $\Cl_q(n, k)$-modules described in \Cref{clqnk repn}, only $\pi_0$ has a classical analog.
To see this, compare \Cref{clqnk action} with the $\Cln$-action on the exterior algebra.
In particular, the operators $\pi_0(\psi_a)$ and $\pi_0(\psi_a^\dg)$ satisfy the canonical anticommutation relations \eqref{cac}.

The action of classical Clifford algebra $\Cln$ on its spinor representation can be described in terms of inner and exterior multiplication.
However, in the quantum Clifford algebra, even under $\pi_0$, the generators $\psi_a^\dg$ and $\psi_a$ no longer act by inner and exterior multiplication because the algebra structure of $\qbigwedge(\field^n)$ is such that a product $v(\ell)$ acquires powers of $\pm q$ when re-arranging factors of $v_k$.
For instance, exterior multiplication of $v(e_1) = v_1$ by $v_2$ results in $-\qinv v(e_1 + e_2) = -\qinv v_1 v_2$.

In order to obtain quantum analogs of the inner and exterior multiplication operators that correctly account for these powers of $\pm q$, we must use the $\omega_k$.
Concretely, we let $\iota_j^q$ and $\varepsilon_j^q$ denote quantum analogs of the inner and exterior multiplication operators defined on any $v(\ell)$ by
\begin{align*}
	\iota_j^q(v(\ell)) &= (-q)^{\ell_1 + \cdots + \ell_{j-1}} v(\ell - e_j)
	&= \pi_0 \bigg(\prod_{k < j} \omega_k \cdot q^{j-1} \psi_i^\dg \bigg)  v(\ell),
	\\
	\varepsilon_j^q (v(\ell)) &= (-\qinv)^{\ell_1 + \cdots + \ell_{j-1}} v(\ell + e_j)
	& = \pi_0 \bigg(\prod_{k < j} \omega_k^{-1} \cdot q^{1 - j} \psi_j \bigg)  v(\ell).
\end{align*}

There is likely a bilinear form on the spinor representation $\pi_0$ of $\Cl_q(n,k)$ that gives rise to the polarization as a $U_q(\g, k)$-module~\cite[Eq.~(2.5.1)]{Kashiwara91}.
This is motivated by the fact that the transpose on $\Cl_q(n,k)$ goes to the star involution on $U_q(\g, k)$ and there is a symmetric bilinear form on the classical Clifford algebra defined by $\inner{x}{y} = \langle x^t y \rangle_0$, where $\langle x \rangle_0$ denotes the degree $0$ component (\textit{i.e.}, constant term) of $x \in \Cln$.

The $\Cl_q(n, k)$-modules defined in \Cref{clqnk repn} are compatible with the isomorphism $\Gamma_q \colon \Cl_q(n, k)^{\otimes m} \to \Cl_q(nm, k)$ of \Cref{clqn clqnm embedding} in the sense that the following diagram commutes:
\begin{equation}\label[diag]{braided ext alg isom as qcl modules}
	\begin{tikzcd}[row sep=small]
		\Cl_q(nm, k) \rar 
		& \End\left(\qbigwedge(V^{(nm)})\right) \\[+0.8em]
		\Cl_q(n, k)^{\otimes m} \rar \uar
		& \End\left(\qbigwedge(\field^n)^{\otimes m}\right). \uar
	\end{tikzcd}
\end{equation}

\begin{prop}
	Let $T\colon \qbigwedge(\field^n)^{\otimes m} \to \qbigwedge(\field^{nm})$ denote the linear map defined by
	\[
	T\left(v(\ell^{(1)}) \otimes \cdots \otimes v(\ell^{(m)})\right) = v\left(\sum_{j=1}^m \sum_{i=1}^n \ell_i^{(j)} e_{i + (j-1)n}\right).
	\]
	For $p \in \mathbb{Z}_{2k}^n$, let $\pi_p^{(r)}$ denote the $\Cl_q(r, k)$-module defined by \Cref{clqnk repn}.
	Then for any $\varphi \in \Cl_q(n, k)^{\otimes m}$,
	\begin{align}\label{gamma is compatible}
		\pi_p^{(nm)}\left(\Gamma_q(\varphi)\right) \circ T 
		= 
		T \circ \left(\pi_p^{(n)}\right)^{\otimes m}(\varphi).
	\end{align}
\end{prop}

\begin{proof}
For simplicity of notation, we assume $p = 0$ without loss of generality.
It suffices to verify \Cref{gamma is compatible} for a set of $\Cl_q(n, k)^{\otimes m}$ generators.
To this end, let $\phi_a$ denote either $\psi_a$ or $\psi_a^\dg$ and set $\phi_{a, j} = 1 \otimes \cdots \otimes \phi_a \otimes \cdots \otimes 1$ as in the proof of \Cref{clqn clqnm embedding}.
Define $\omega_{a, j}$ analogously and notice \Cref{clqnk repn} implies
\begin{align*}
	\pi_0\left(\Gamma_q(\omega_{a, j})\right) &T\left(v(\ell^{(1)}) \otimes \cdots \otimes v(\ell^{(m)})\right)
		= \omega_{a + (j-1)n} v\left(\sum_{k=1}^m \sum_{i=1}^n \ell_i^{(k)} e_{i + (k-1)n}\right) \\
		&= T\left(v(\ell^{(1)}) \otimes \cdots \otimes 
		   \left(q^{\ell_a^{(j)}-1} v(\ell^{(j)})\right) 
		   \otimes \cdots \otimes v(\ell^{(m)})\right) \\
		&= T\left(\left(\pi_0^{(n)}\right)^{\otimes m} \left(\omega_{a,j}\right) v(\ell^{(1)}) \otimes \cdots \otimes v(\ell^{(m)})\right).
\end{align*}

To conclude, we take $\varphi = \phi_{a, j}$.
We will need the action of the volume element $f_r$ on $\qbigwedge(\field^n)$.
\Cref{psi psid comm anticomm} implies that each (non-central) commutator $\ppdcomm$ acts diagonally:
\begin{align*}
	\pi_p\left(\ppdcomm\right) v(\ell) 
		&= \frac{q^{k(1 - \ell_a)} - q^{k \ell_a}}{q^k - 1} \, v(\ell) 
		= (-1)^{\ell_a} v(\ell).
\end{align*}
Together with \Cref{clqnk repn} and \Cref{volume elt}, the last calculation implies the desired relation:
\begin{align*}
	\pi_0\left(\Gamma_q(\phi_{a, j})\right) &T\left(v(\ell^{(1)}) \otimes \cdots \otimes v(\ell^{(m)})\right)
		= f_{(j-1)n} \phi_{a + (j-1)n} v\left(\sum_{k=1}^m \sum_{i=1}^n \ell_i^{(k)} e_{i + (k-1)n}\right) \\
		&= (-1)^{\sum_{i=1}^{a-1} \ell_i^{(j)}} v\left(\sum_{k=1}^m \sum_{i=1}^n \ell_i^{(k)} e_{i + (k-1)n} \pm e_{a + (j-1)n}\right) \\
		&= T\left(v(\ell^{(1)}) \otimes \cdots \otimes \pi_0^{(n)}\left(\phi_a\right) v(\ell^{(j)}) \otimes \cdots \otimes v(\ell^{(m)})\right) \\
		&= T\left(\left(\pi_0^{(n)}\right)^{\otimes m} \left(\phi_{a,j}\right) v(\ell^{(1)}) \otimes \cdots \otimes v(\ell^{(m)})\right). 
		\qedhere
\end{align*}
\end{proof}

Notice that no irreducible module of $\Cl_q(n,k)$ is faithful since
\[
\dim \Cl_q(n,k) = (8k)^n = 2^{2n} \cdot (2k)^n > 2^{2n} = \dim \End\left( \qbigwedge(\field^n) \right).
\]
This can be explained by the nontrivial center of $\Cl_q(n,k)$, which accounts for the difference in dimension from its classical counterpart $\Cln$.

Despite this, we may combine the inequivalent $\pi_p$ to prove $\Cl_q(n, k)$ is isomorphic to a direct sum of matrix algebras.
This means that there are no other irreducible representations other than those constructed in \Cref{clqnk repn}.
For contrast, recall that the classical representation $\Cln \to \End(\bigwedge(V))$ is faithful, so there is an isomorphism $\Cln \iso \mathrm{Mat}_{2^n}(\field)$ showing $\Cln$ is simple.

\begin{thm}\label{clqnk is semisimple}
	If $\zeta_{2k} \in \field$, then the quantum Clifford algebra $\Cl_q(n, k)$ is semisimple.
	In particular, 
	\[
	\Cl_q(n, k) \cong \bigoplus_{p \in \ZZ_{2k}^n} \End\left(\Omega^p\right) \cong (2k)^n \mathrm{Mat}_{2^n} (\field).
	\]
\end{thm}

\begin{proof}
	Each $\pi_p$ is irreducible, so it surjects onto $\End(\Omega^p) \cong \mathrm{Mat}_{2^n}(\field)$. Therefore, the sum $\bigoplus_p \pi_p$, which maps $\Cl_q(n, k)$ into a matrix algebra of dimension $(2k)^n \cdot (2^n)^2 = (8k)^n = \dim(\Cl_q(n, k))$, must be an isomorphism.
\end{proof}

When $\zeta_{2k} \notin \field$, the semisimplicity of $\Cl_q(n, k)$ depends on the base field $\field$ over which it is defined.
In particular, $\Cl_q(n, k)$ is semisimple exactly when the group ring $\field [\ZZ_{2k}^n] \iso Z_q(n, k)$ is semisimple. Since $\Cl_q(n, k) \iso \Cl_q(1, k)^{\otimes n}$ by \Cref{clqn is n fold product of cl1}, we need only consider the semisimplicity of $Z_q(1, k)$.

\begin{thm}\label{generalized semisimplicity}
	The quantum Clifford algebra $\Cl_q(n, k)$ is semisimple if and only if $Z_q(1, k) \iso \field [\ZZ_{2k}]$ is semisimple.
\end{thm}

\begin{proof}
We consider a slightly more general problem.
Let $R$ be a commutative ring, and fix some invertible $r \in R^{\times}$.
Let $\Cl_r(R \oplus R^*)$ be the Clifford algebra generated by $\psi, \psi^*$ subject to $\psi \psi^* + \psi^* \psi = r$.
We show that if $R$ is Artinian, then $\Cl_r(R \oplus R^*)$ is semisimple if and only if $R$ is semisimple.
The claim then follows from \Cref{clqnk center} and \Cref{clqn is n fold product of cl1}.

We first show there is a one-to-one correspondence between the two-sided ideals of $\Cl_r(R \oplus R^*)$ and $R$.
Clearly any ideal $J \subseteq R$ can be used to construct an ideal of $\Cl_r(R \oplus R^*)$.
Thus, we just need to show that for every two-sided ideal $I \subseteq \Cl_r(R \oplus R^*)$, 
we have the same form as above for some ideal $J \subseteq R$.
It is sufficient to show this for a principal ideal of $\Cl_r(R \oplus R^*)$.

We start with a generic element
\[
x = \alpha + \beta \psi + \gamma \psi^* + \delta \psi \psi^*
\]
in $\Cl_r(R \oplus R^*)$ and we look at the two-sided ideal $I$ generated by $x$.
We will show by direct computation that this is equivalent to the ideal $(\alpha, \beta, \gamma, \delta) \subseteq R$.
The following elements belong to $I$:
\begin{align*}
\alpha \psi & = r^{-1} \psi x \psi^* \psi,
&
\alpha \psi^* & = r^{-1} \psi^* \psi x \psi^*,
&
\alpha \psi \psi^* & = \psi x \psi^*,
\\
\beta \psi & = r^{-2} \psi \psi^* x \psi^* \psi,
&
\beta \psi^* & = r^{-1} \psi^* x \psi^*,
&
\beta \psi \psi^* & = r^{-1} \psi \psi^* x \psi^*,
\\
\gamma \psi & = r^{-1} \psi x \psi,
&
\gamma \psi^* & = r^{-2} \psi^* \psi x \psi \psi^*,
&
\gamma \psi \psi^* & = r^{-1} \psi x \psi \psi^*.
\end{align*}
Consequently, we have that $\alpha + \delta \psi \psi^* \in I$, and since $(\alpha + \delta \psi \psi^*) \psi = (\alpha + r \delta) \psi$, we have that $\delta \psi \in I$.
It is easy to see that $\delta \psi^*, \delta \psi \psi^* \in I$, which in turn implies $\alpha \in I$.
Finally, we have $r^{-1}(\psi^* \beta \psi - \beta \psi \psi^*) = \beta \in I$, and similarly $\gamma, \delta \in I$.

The Jacobson radical of a (not necessarily commutative) ring $S$ is the intersection of all maximal left ideals of $S$.
As maximal two-sided ideals are necessarily maximal left ideals, we have that the Jacobson radical of $S$ is contained in the intersection of all two-sided ideals of $S$.
Thus if the Jacobson radical of $R$ is $0$, then the Jacobson radical of $\Cl_r(R \oplus R^*)$ is also $0$.
Since $R$ is Artinian, $\Cl_r(R \oplus R^*)$ is also Artinian as it is a finite rank free module over $R$.
An Artinian ring is semisimple if and only if its Jacobson radical is trivial; see \textit{e.g.},~\cite[Thm.~4.14]{Lam01}.
Therefore, if $R$ is semisimple then $\Cl_r(R \oplus R^*)$ is semisimple as well.

For the converse we assume $R$ is not a semisimple ring, and so $R$ is not a semisimple (rank $1$ free) $R$-module.
Consider the left $\Cl_r(R \oplus R^*)$ spinor module $\pi$, with basis $1, \psi$. Then $\pi \cong R + \psi R$ with $\psi^* R = 0$ and $\psi^* \psi R = r R$.
Since $r \in R^{\times}$, we have $rM = M$ for every $R$-module $M$, and hence, $\Cl_r(R \oplus R^*)$ is not a semisimple ring.
\end{proof}

From the perspective that $\Cl_q(n,k) \iso \Cl_r(R \oplus R^*)^{\otimes n}$ for $R = Z_q(n, k)$, we can construct all irreducible representations when $R$ is semisimple by inducing an irreducible $Z_q(n, k)$-representation.
This allows us to remove the condition that $\zeta_{2k} \in \field$ from \Cref{clqnk repn} and \Cref{clqnk is semisimple}, although the precise descriptions stated there change and involve the representation theory of $\ZZ_{2k}$.
For instance, if $n = 1$, $k = 2$, and $\field = \RR$, we could induce from the $2$ dimensional (nonsplit) irreducible rotation $\ZZ_4$-module (so $z_1$ acts by rotating $\RR^2$ by $\pi/2$).
We thank Gurbir Dhillon for pointing out that these irreducible representations could also be constructed directly by inducing from an irreducible representation of the commutative ring $\langle \omega_1, \dotsc, \omega_n \rangle$.

We conclude this section by showing that the quantum volume element has important representation-theoretic implications, much like its classical counterpart.
For example, there is a quantum analog of~\cite[Prop.~5.10]{michelsohn_lawson}. 

\begin{prop}
	Consider the superalgebra grading on $\Cl_q(n, k)$, and let $\Cl_q^+(n, k)$ denote the subalgebra of $\Cl_q(n, k)$ generated by elements of even degree.
	Let $\pi\colon \Cl_q(n, k) \to \End(W)$ be an irreducible representation of $\Cl_q(n,k)$.
	Then there exists a splitting
	\[
	W = W^+ \oplus W^-,
	\]
	with $W^\pm = \bigl(1 \pm \pi(f_n) \bigr)(W)$, where $W^\pm$ is invariant under $\Cl_q^+(n, k)$.
\end{prop}

\begin{proof}
	\Cref{volume elt props} implies $f_n$ commutes with $\Cl_q^+(n, k)$, so $\Cl_q^+(n, k)$ preserves every eigenspace of $f_n$.
	The claim then follows because $f_n^2 = 1$: $W^\pm$ is a $\pm 1$ eigenspace of $f_n$ and the only possible eigenvalues of $f_n$ are $-1$ and $+1$.
\end{proof}

\section{Twisting and square roots}
\label{sec:root}

We note that nearly all the relations defining $\Cl_q(n, k)$ allow us to take $k = 1/2$ except for \Cref{eq:replaced_qcomm_rels}.
Yet, we can adjust this by considering new generators
\[
\phi_a = \psi_a
\qquad\text{and}\qquad
\phi_a^* = \psi_a^* \omega_a^k.
\]
Using these generators, our algebra is presented by
\begin{equation}
\label{eq:defining_half_rels}
\begin{aligned}
\omega_a \omega_b & = \omega_b \omega_a,  \hspace{70pt}
& \omega_a^{4k} & = (1 + q^{-2k}) \omega_a^{2k} - q^{-2k},
\\ \omega_a \phi_b & = q^{\delta_{ab}} \phi_b \omega_a,
& \omega_a \phi_b^* & = q^{-\delta_{ab}} \phi_b^* \omega_a,
\\ \phi_a \phi_b & + \phi_b \phi_a = 0,
& \phi_a^* \phi_b^* & + \phi_b^* \phi_a^* = 0,
\\ \phi_a \phi_a^* & + \phi_a^* \phi_a = 1,
& \phi_a \phi_a^* & + q^{-2k} \phi_a^* \phi_a = \omega_a^{2k},
\\ \phi_a \phi_b^* & + \phi_b^* \phi_a = 0 && \text{if } a \neq b.
\end{aligned}
\end{equation}

Using this presentation, we can take $k \in \frac{1}{2} \ZZ_{>0}$.
It is easy to compute analogous reduction rules involving $\phi_a$ and $\phi_a^{\dg}$ instead of $\psi_a$ and $\psi_a^{\dg}$.
In addition to those from a simple reformulation of the defining relations, we have the following reductions analogous to \Cref{generic reduction relations}.

\begin{lemma}
\label[lem]{lem:half_twist_relations}
The following relations hold in $\Cl_q(n,k)$:
\begin{equation}
\omega_a^{2k} = (1 - q^{-2k}) \phi_a \phi_a^{\dg} + q^{-2k},
\qquad\qquad
\phi_a^{\dg} \phi_a \phi_a^{\dg} = \phi_a^{\dg},
\qquad\qquad
\phi_a \phi_a^{\dg} \phi_a \phi_a^{\dg} = \phi_a \phi_a^{\dg}.
\end{equation}
\end{lemma}

Thus, we have the analog of \Cref{thm:basis_general} by a similar proof.

\begin{thm}
\label{thm:basis_half_k}
Consider the set $\mcB'$ of monomials
\[
	\prod_{a=1}^n \phi_a^{p_a} (\phi_a^*)^{d_a} \omega_a^{v_a},
\]
where $p_a, d_a \in \{0, 1\}$ and $v_a \in \{0, 1, \dotsc, 2k-1\}$.
Then $\mcB'$ is a basis of $\Cl_q(n, k)$ for $k \in \frac{1}{2} \ZZ_{>0}$.
\end{thm}

Therefore, the dimension formula from \Cref{clqnk dim} still holds.
In particular, we have $\dim \Cl_q(n, 1/2) = 4^n$.

We also have the analog of \Cref{thm:basis_alt} when $q^{2k} \neq 1$.
Furthermore, when $q^{2k} \neq 1$, we also note that~\eqref{eq:replaced_qcomm_rels} and \Cref{eq:reduction_v4} imply
\begin{align}
	\label{eq:phi_qcomm_rels}
	\phi_a \phi_a^* = \frac{\omega_a^{2k} (q^{2k} \omega_a^{2k} - 1)}{q^{2k} - 1} = \frac{q^{2k} \omega_a^{2k} - 1}{q^{2k}-1},
	\qquad
	\phi_a^* \phi_a = \frac{q^{4k} \omega_a^{2k} (1 - \omega_a^{2k})}{q^{2k} - 1} = \frac{q^{2k}(1 - \omega_a^{2k})}{q^{2k} - 1},
\end{align}
yielding another alternative relation
\[
	\phi_a \phi_a^* + q^{-4k} \phi_a^* \phi_a = \omega_a^{4k} = (1 + q^{-2k}) \omega_a^{2k} - q^{-2k}.
\]

It is also easy to see that the center computation (\Cref{clqnk center}) extends to the case $k \in \ZZ_{\geq 0} + \frac{1}{2}$ as \Cref{eq:center_gen} becomes
\[
	z_a = q \omega_a - (q-1) \phi_a \phi_a^{\dg} \omega_a.
\]
In particular, when $k = \frac{1}{2}$, we have
\begin{equation}
\label{eq:omega_end}
\omega_a = (1 - q^{-1}) \phi_a \phi_a^{\dg} + q^{-1}
\end{equation}
and $z_a = 1$.
Thus, the center of $\Cl_q(n, \frac{1}{2})$ is trivial as claimed by \Cref{clqnk center}.

Let us further explore the case when $k = \frac{1}{2}$.
As we saw, these quantum Clifford algebras do not have a nontrivial center and start looking like the classical Clifford algebra as \Cref{eq:omega_end} says $\omega_a$ belongs to the subalgebra generated by $\psi_a$, $\psi_a^*$.
In fact, by examining the relations in \Cref{lem:half_twist_relations}, we see that $\Cl_q(n, \frac{1}{2})$ can be given the presentation of $\Cln$.

\begin{thm}
\label{thm:classical_iso}
We have
\[
	\Cl_q(n, 1/2) \iso \Cln.
\]
\end{thm}

\Cref{thm:classical_iso} is also another reflection of the principle that the quantum Clifford algebra is a classical Clifford algebra over its (group ring) center.

With a little bit of work to express everything in terms of the new generators, it can be shown that \Cref{clqn clqnm embedding}, \Cref{clqnk is semisimple}, and \Cref{generalized semisimplicity} still hold for all $k \in \frac{1}{2} \ZZ_{>0}$.
In particular, we build the modified quantum volume element $\check{f}_r = [\phi_1, \phi_1^{\dg}] \cdots [\phi_r, \phi_r^{\dg}]$, which behaves analogously to the classical volume element (or \Cref{volume elt props}) by the defining relations~\eqref{eq:defining_half_rels}.
We leave the details for the interested reader.

The quantum Clifford algebra for $k = \frac{1}{2}$ has also appeared in~\cite{DF94} in connection with an FRT construction, where we note that~\cite[Eq.~(5.3.1)--(5.3.4)]{DF94} are the defining relations (using \Cref{lem:half_twist_relations} for rewriting $\omega_a^{2k}$).
In particular, Ding and Frenkel~\cite{DF94} start with Theorem~\ref{thm:classical_iso} and construct $\Cl_q(n, \frac{1}{2})$.

We could consider the limit $k \to \infty$ (in contrast to the case $n \to \infty$ considered in~\cite{Hayashi90}) when, say, $\abs{q} > 1$.
In this case our center tends towards a product of circle groups, $\omega_a^{2k} \to 1$, any spinor representation is no longer irreducible (although it remains indecomposible), and its ``dual'' (in the sense of \Cref{rmk:dual_spinor_repr}) is $1$ dimensional.
Indeed, we would have $\psi_a^{\dg} \psi_a v = 0$ (or $\phi_a^{\dg} \phi_a v = 0$) for any $v$ in any spinor representation.

\section{Clifford algebra morphisms}
\label{sec:Clifford_relationship}

We examine the remark by Mitsuhiro Takeuchi included at the end of~\cite{Hayashi90}.
We assume that $q^{2k} \neq 1$ and $\zeta_{2k} \in \field$.
For brevity, we set $\Cl_{(n)} = \Cln$.
We denote the generators of $\Cl_{(n)}$ by $v_i, v_i^{\dg}$, for $i = 1, \dotsc, n$.

\begin{prop}[M.~Takeuchi~\cite{Hayashi90}]
\label[prop]{prop:clifford_iso_Takeuchi}
Suppose $\zeta_{2k} \in \field$.
Consider the map $\vartheta \colon \Cl_q(n, k) \to \Cl_{(n)}^{(2k)^n}$ defined on any $y \in \Cl_q(n,k)$ by
\[
\vartheta(m) = \bigl( \vartheta_{\zz}(y) \bigr)_{\zz \in \{\zeta_{2k}^j \mid j=1,\dotsc,2k\}^n}
\]
where for each $\zz \in \{\zeta_{2k}^j \mid j =1,\dotsc,2k\}$ the $\ZZ^n$-graded $\field$-algebra morphism $\vartheta_{\zz} \colon \Cl_q(n, k) \to \Cl_{(n)}$ is defined on generators by
\[
\vartheta_{\zz}(\psi_i) = z_i v_i,
\qquad
\vartheta_{\zz}(\psi_i^{\dg}) = z_i v_i^{\dg},
\qquad
\vartheta_{\zz}(\omega_1) = z_i(v_i v_i^{\dg} + q^{-1} v_i^{\dg}v_i) = z_i (1 - q^{-1}) v_i v_i^{\dg} + z_i q^{-1}.
\]
Then $\vartheta$ is a $\field$-algebra isomorphism.
\end{prop}

\begin{proof}
By either directly noting the compatible anticommutator relations or using our structure results (\Cref{cl tensor m into clnm} and \Cref{clqn clqnm embedding}), it is sufficient prove the $n = 1$ case.
Since $(v_i v_i^{\dg})^2 = v_i v_i^{\dg}$, by a straightforward induction argument, we have
\[
\vartheta_{\zz}(\omega_i)^m = z_i^m (1 - q^{-m}) v_i v_i^{\dg} + z_i^m q^{-m}.
\]
By verifying the relations of $\Cl_q(1,k)$, we see that $\vartheta_{\zz}$ is a surjective $\ZZ^n$-graded $\field$-algebra morphism for any fixed $\zz$ .

Now we show $\vartheta$ is surjective.
To this end, let $\varphi := \vartheta(\psi_1)$, $\varphi^{\dg} := \vartheta(\psi_1^{\dg})$, and $\varpi := \vartheta(\omega_1)$.
We fix the order of the sequence defining $\vartheta$ as $(1, \zeta_{2k}, \dotsc, \zeta_{2k}^{2k-1})$.
Then, we have
\begin{equation}
\label{eq:rotation_gens}
q^m \varphi \varpi^m = (\zeta_{2k}^{j(m+1)} v_1)_{j=0}^{2k-1},
\qquad\qquad
q^m \varpi^m \varphi^{\dg} = (\zeta_{2k}^{j(m+1)} v_1^{\dg})_{j=0}^{2k-1}.
\end{equation}
From basic linear algebra (or the representation theory of $\ZZ_{2k}$) as this is essentially giving a discrete Fourier transform, we can obtain $v_{(j)} := (\delta_{jm} v_1)_{m=0}^{2k-1}$ and $v_{(j)}^{\dg} := (\delta_{jm} v_1^{\dg})_{m=0}^{2k-1}$, for any fixed $j = 0, \dotsc, 2k-1$, through linear combinations of the elements~\eqref{eq:rotation_gens}.
Thus, we have obtained all of the generators of $\Cl_{(1)}^{2k}$, and hence $\vartheta$ is surjective.

We can define the inverse map $\vartheta^{-1}$ by the inverse discrete Fourier transform
\[
\vartheta(v_{(j)}) = \frac{1}{2k} \sum_{m=0}^{2k-1} \zeta_{2k}^{-j(m+1)} q^m \psi_1 \omega_1^m,
\qquad\qquad
\vartheta(v_{(j)}^{\dg}) = \frac{1}{2k} \sum_{m=0}^{2k-1} \zeta_{2k}^{-j(m+1)} \psi_1^{\dg} \omega_1^m.
\]
A direct computation shows this is a $\field$-algebra morphism and the inverse of $\vartheta$.
\end{proof}

In particular, we note that $\dim \Cl_{(n)}^{(2k)^n} = (2k)^n 4^n = (8k)^n$ as there are $(2k)^n$ such tuples $\zz$.
\Cref{prop:clifford_iso_Takeuchi} gives an alternative proof that $\mcB$ is linearly independent (without relying on Theorem~\ref{thm:basis_general}).
The appearance of the discrete Fourier transform comes from the fact that we can consider the quantum Clifford algebra $\Cl_q(n, k)$ as a Clifford algebra over the center $Z_q(n, k)$ (see \Cref{generalized semisimplicity}).
Indeed, we are effectively splitting the natural adjoint representation of the center into its irreducible ($1$ dimensional) modules and taking the corresponding Clifford algebra $\Cl(\field \oplus \field^*)$ on each such module.
Compare this with \Cref{clqnk is semisimple}.

\begin{ex}
Consider the case $n = 1$ and $k = 1$.
The isomorphism $\vartheta \colon \Cl_q(1, 1) \to \Cl_{(1)} \times \Cl_{(1)}$ from \Cref{prop:clifford_iso_Takeuchi} is given explicitly on basis elements by
\begin{align*}
1 & \mapsto (1, 1),
& \omega_1^m & \mapsto \bigl( (1-q^{-m}) v_1 v_1^{\dg} + q^{-m}, (-1)^m(1-q^{-m}) v_1 v_1^{\dg} + (-1)^m q^{-m} \bigr),
\\ \psi_1 & \mapsto (v_1, -v_1),
&\psi_1 \omega_1 & \mapsto (q^{-1} v, q^{-1} v_1),
\\ \psi_1^{\dg} & \mapsto (v_1^{\dg}, -v_1^{\dg}),
& \psi_1^{\dg} \omega_1 & \mapsto (v_1^{\dg}, v_1^{\dg}),
\end{align*}
for $m = 1, 2, 3$.
For fixed $x = \pm 1$, the morphism $\vartheta_x$ is surjective but not injective since $\vartheta_1(\psi_1^{\dg} - x \psi_1^{\dg} \omega_1) = 0$.
However, looking at, \textit{e.g.}, $(\psi_1^{\dg}, \psi_1^{\dg} \omega_1)$, we get the discrete Fourier transform matrix $\begin{bmatrix} 1 & 1 \\ -1 & 1 \end{bmatrix}$.
We can see that
\begin{align*}
\vartheta^{-1}(v_1, 0) & = \psi_1 - q \psi_1 \omega_1,
&
\vartheta^{-1}(v_1^{\dg}, 0) & = \psi_1^{\dg} - \psi_1^{\dg} \omega_1,
\\
\vartheta^{-1}(0, v_1) & = \psi_1 + q \psi_1 \omega_1,
&
\vartheta^{-1}(0, v_1^{\dg}) & = \psi_1^{\dg} + \psi_1^{\dg} \omega_1,
\end{align*}
is coming from the inverse discrete Fourier transform matrix.
\end{ex}

We continue to examine the case $n = 1$ and $k = 1$ in a bit more detail.
Let $\Cl_{a,b}$ denote the Clifford algebra given by the quadratic form with signature $(a, b)$; that is, the form whose Hessian matrix is diagonal with $a$ entries being $1$ followed by $b$ entries being $-1$.
We can change the basis in $\Cl_{(1)} := \langle v, v^{\dg} \rangle$ by writing
\begin{equation}
\label{eq:usual_Clifford_COB}
\xi = v + v^{\dg},
\qquad\qquad
\xi^{\dg} = v - v^{\dg},
\end{equation}
which satisfy
\[
\xi^2 = 1,
\qquad
(\xi^{\dg})^2 = -1,
\qquad
\xi \xi^{\dg} = 1 - 2 v v^{\dg},
\qquad
\xi \xi^{\dg} + \xi^{\dg}\xi = 0.
\]
(Note $\xi = \epsilon_2$ and $\xi^{\dg} = \epsilon_1$ from \Cref{classical std coords}.)
Thus, $\Cl_{(1)} \iso \Cl_{1,1}$, where the inverse isomorphism is given by
\begin{equation}
\label{eq:inverse_COB}
v = \frac{\xi + \xi^{\dg}}{2},
\qquad\qquad
v^{\dg} = \frac{\xi - \xi^{\dg}}{2},
\qquad\qquad
v v^{\dg} = \frac{1 - \xi \xi^{\dg}}{2}.
\end{equation}
Next, we have $\Cl_{2,1} \iso \Cl_{1,1} \times \Cl_{1,1}$ as $\field$-algebras by the isomorphism
\[
e_1 \mapsto (\xi, \xi),
\qquad\qquad
e_2 \mapsto (\xi \xi^{\dg}, -\xi \xi^{\dg}),
\qquad\qquad
e_3 \mapsto (\xi^{\dg}, \xi^{\dg}),
\]
where $e_1^2 = e_2^2 = 1$ and $e_3^2 = -1$.
Note that this is not a filtered $\field$-algebra isomorphism under the natural filtrations.
Hence, we have an isomorphism $\Cl_q(1, 1) \iso \Cl_{2,1}$ as $\field$-algebras.
We can write out this isomorphism explicitly:
\begin{gather*}
\begin{aligned}
\frac{e_1 + e_3}{2} & = (v, v) = q \psi_1 \omega_1,
&
\frac{e_1 - e_3}{2} & = (v^{\dg}, v^{\dg}) = \psi_1^{\dg} \omega_1,
\\
\frac{e_1 e_2 + e_2 e_3}{2} & = (-v^{\dg}, v^{\dg}) = -\psi_1,
&
\frac{e_1 e_2 - e_2 e_3}{2} & = (v, -v) = \psi_1,
\end{aligned}
\\
\begin{aligned}
\omega_1^m & = \left( \frac{(q^{-m} - 1) \xi \xi^{\dg} + q^{-m} + 1}{2}, 
 (-1)^m \frac{(q^{-m} - 1) \xi \xi^{\dg} + q^{-m} + 1}{2} \right)
\\ & = \begin{cases}
\frac{q^{-m} - 1}{2} e_1 e_3 + \frac{q^{-m}+1}{2} & \text{if } m = 2, \\
\frac{q^{-m} - 1}{2} e_2 - \frac{q^{-m}+1}{2} e_1 e_2 e_3 & \text{otherwise}.
\end{cases}
\end{aligned}
\end{gather*}

For general $n$, with $k = 1$, we note that all of the other relations are supercommuting relations with considering $\omega_a$ as an even element and $\psi_a$ and $\psi_a^{\dg}$ as odd elements.
Therefore, we can form an isomorphism by taking the $n$-fold supersymmetric product of $\Cl_{2,1}$ with $e_2$ being an even element and $e_1, e_3$ being odd elements.
However, this is not an isomorphism of supercommutative superalgebras.

\bibliographystyle{alphaurl}
\bibliography{ref}
\end{document}